\documentclass[
final
 , nomarks,pdftex,breaklinks=true,colorlinks=true
]{dmtcs-episciences}

\usepackage[utf8]{inputenc}

\usepackage{pgf,tikz,ifthen}
\usetikzlibrary{arrows}
\usetikzlibrary{calc}
\usepackage{enumerate} 
\usepackage{amsthm,amsmath}

\allowdisplaybreaks

\usepackage[square,numbers]{natbib}

\renewcommand{\deg}{\operatorname{deg}}
\newcommand{\T}{\mathcal{T}}

\theoremstyle{plain}
\newtheorem{theorem}{Theorem}
\newtheorem*{theorem*}{Theorem}
\newtheorem{lemma}{Lemma}
\newtheorem*{proposition*}{Proposition}
\newtheorem*{corollary*}{Corollary}
\newtheorem{ctc}{Conway's Thrackle Conjecture}

\theoremstyle{remark}
\newtheorem*{remark*}{Remark}

\author{Grace Misereh\affiliationmark{1} \and Yuri Nikolayevsky}

\title[Annular and pants thrackles]{Annular and pants thrackles}

\affiliation{Department of Mathematics and Statistics, La Trobe University, Melbourne, Australia}
  

\received{2017-8-17}
\revised{2018-4-3}
\accepted{2018-4-27}

\begin{document}
\publicationdetails{20}{2018}{1}{16}{3883}
\maketitle
\begin{abstract}
A thrackle is a drawing of a graph in which each pair of edges meets precisely once. Conway's Thrackle Conjecture asserts that a thrackle drawing of a graph on the plane cannot have more edges than vertices. We prove the Conjecture for thrackle drawings all of whose vertices lie on the boundaries of $d \le 3$ connected domains in the complement of the drawing. We also give a detailed description of thrackle drawings corresponding to the cases when $d=2$ (annular thrackles) and $d=3$ (pants thrackles).
\end{abstract}

\section{Introduction}
\label{s:intro}

Let $G$ be a finite simple graph with $n$ vertices and $m$ edges. A \emph{thrackle drawing} of $G$ on the plane is a drawing $\T:G\rightarrow\reals^2$, in which every pair of edges meets precisely once, either at a common vertex or at a point of proper crossing (see \cite{LPS97} for definitions of a drawing of a graph and a proper crossing). The notion of thrackle was introduced in the late sixties by John Conway, in relation with the following conjecture. 

\begin{ctc}
For a thrackle drawing of a graph on the plane, one has $m\leq n$.
\end{ctc}

Despite considerable effort \cite{WOO71, LPS97, GY2000, GMY2004, GY2009, PJS2011, FP2011, GY2012, GKY2015, GX2017, FP17, MNajc}, the conjecture remains open. The best known bound for a thrackleable graph with $n$ vertices is $m \le 1.3984 n$ \cite{FP17}. 

Adding a point at infinity we can consider a thrackle drawing on the plane as a thrackle drawing on the $2$-sphere $S^2$. The complement of a thrackle drawing on $S^2$ is the disjoint union of open discs. We say that a drawing \emph{belongs to the class $T_d, \; d \ge 1$}, if there exist $d$ open discs $D_1, \dots, D_d$ whose closures are pairwise disjoint such that all the vertices of the drawing lie on the union of their boundaries (a disk may contain no vertices on its boundary). We say that two thrackle drawings of class $T_d$ are \emph{isotopic} if they are isotopic as drawings on $S^2 \setminus (\cup_{k=1}^d D_k)$. We will also occasionally identify a graph $G$ with its thrackle drawing $\T(G)$ speaking, for example, of the vertices and edges of the drawing.

Thrackles of class $T_1$ are called \emph{outerplanar}: all their vertices lie on the boundary of a single disc $D_1$. Such thrackles are very well understood.

\begin{theorem} \label{t:outer}
Suppose a graph $G$ admits an outerplanar thrackle drawing. Then
\begin{enumerate}[{\rm (a)}]
  \item \label{it:outallodd}
  any cycle in $G$ is odd \cite[Theorem~1]{GY2012};

  \item \label{it:{it:outCTC}}
  the number of edges of $G$ does not exceed the number of vertices \cite[Theorem~2]{PJS2011};

  \item \label{it:outRei}
  if $G$ is a cycle, then the drawing is Reidemeister equivalent to a standard odd musquash \cite[Theorem~1]{GY2012}.
\end{enumerate}
\end{theorem}

We say that two thrackle drawings are \emph{Reidemeister equivalent} (or \emph{equivalent up to Reidemeister moves}), if they can be obtained from one another by a finite sequence of Reidemeister moves of the third kind in the complement of vertices (see Section~\ref{ss:R}).

A \emph{standard odd musquash} is the simplest example of a thrackled cycle: for $n$ odd, distribute $n$ vertices evenly on a circle and then join by an edge every pair of vertices at the maximal distance from each other. This defines a musquash in the sense of Woodall \cite{WOO71}: an \emph{$n$-gonal musquash} is a thrackled $n$-cycle whose successive edges $e_0,\dots,e_{n-1}$ intersect in the following manner: if the edge $e_0$ intersects the edges $e_{k_1},\dots,e_{k_{n-3}}$ in that order, then for all $j=1,\dots,n-1$, the edge $e_j$ intersects the edges $e_{k_1+j},\dots,e_{k_{n-3}+j}$ in that order, where the edge subscripts are computed modulo $n$. A complete classification of musquashes was obtained in  \cite{GD1999,GD2001}: every musquash is either isotopic to a standard $n$-musquash, or is a thrackled six-cycle.

In this paper, we study thrackle drawings of the next two classes $T_d$: annular thrackles and pants thrackles.

A thrackle drawing of class $T_2$ is called \emph{annular}. Up to isotopy, we can assume that the boundaries of $D_1$ and $D_2$ are two concentric circles on the plane, and that the thrackle drawing, except for the vertices, entirely lies in the open annulus bounded by these circles. Clearly, any outerplanar drawing can be viewed as an annular drawing. Figure~\ref{figure:annulus} shows an example of an annular thrackle drawing which is not outerplanar. Note however that the underlying graph has some vertices of degree $1$ (which must always be the case by Theorem~\ref{t:ann}\eqref{it:annout} below). 
\begin{figure}[h]
\centering
\begin{tikzpicture}[scale=0.5,>=triangle 45]
\draw[thick]  (0,0) circle (4); \draw[thick] (0,0) circle (1); 
\coordinate (A3) at (0,-1);
\coordinate (A5) at (-1,0);
\coordinate (A2) at ({4*cos((pi/6) r)},{4*sin((pi/6) r)});
\coordinate (A1) at ({4*cos((5*pi/6) r)},{4*sin((5*pi/6) r)});
\coordinate (A4) at ({4*cos((11*pi/12) r)},{4*sin((11*pi/12) r)});
\coordinate (A6) at ({4*cos((pi/4) r)},{4*sin((pi/4) r)});
\fill (A1) circle (6.0pt);
\fill (A2) circle (6.0pt);
\fill (A3) circle (6.0pt);
\fill (A4) circle (6.0pt);
\fill (A5) circle (6.0pt);
\fill (A6) circle (6.0pt);
\draw [very thick] (A1) to [out=-100,in=-100] (A3) to [out=-80,in=-80] (A2) to (A1);
\draw [very thick] (A2)--(A4);
\draw [very thick] (A6) to [out=-90,in=0] (0,-2.5) to [out=180,in=-110] (A5);
\draw [very thick] (A2) to [out=-80,in=0] (0,-3.5) to [out=180,in=-120] (A5);
\end{tikzpicture}
\caption{An annular thrackle drawing.}
\label{figure:annulus}
\end{figure}
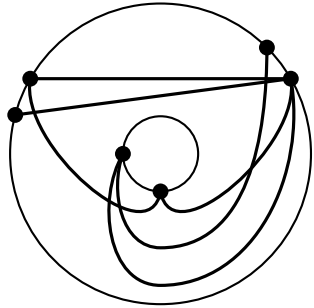

We show that the three assertions of Theorem~\ref{t:outer} also hold for annular drawings.

\begin{theorem} \label{t:ann}
Suppose a graph $G$ admits an annular thrackle drawing. Then
\begin{enumerate}[{\rm (a)}]
  \item \label{it:annodd}
  any cycle in $G$ is odd;

  \item \label{it:annC}
  the number of edges of $G$ does not exceed the number of vertices;

  \item \label{it:annout}
  if $G$ is a cycle, then the drawing is, in fact, outerplanar \emph{(}and as such, is Reidemeister equivalent to a standard odd musquash\emph{)}.
\end{enumerate}
\end{theorem}

We next proceed to the thrackle drawings of class $T_3$. We call such drawings \emph{pants thrackle drawings} or \emph{pants thrackles}.

Any annular thrackle drawing is trivially a pants thrackle drawing. The pants thrackle drawing of a six-cycle in Figure~\ref{figure:sixcycle} is not annular.
\begin{figure}[h]
\centering
\begin{tikzpicture}[scale=0.6,>=triangle 45]
\draw[thick]  (0,0) ellipse (5 and 3); \draw[thick]  (-2,0) circle (1); \draw[thick]  (2,0) circle (1);
\coordinate (A1) at ({5*cos((10*pi/9) r)},{3*sin((10*pi/9) r)});
\coordinate (A2) at ({5*cos((8*pi/9) r)},{3*sin((8*pi/9) r)});
\coordinate (B1) at ({-2+cos((-pi/9) r)},{sin((-pi/9) r)});
\coordinate (B2) at ({-2+cos((pi/9) r)},{sin((pi/9) r)});
\coordinate (C2) at ({2+cos((pi/9) r)},{sin((pi/9) r)});
\coordinate (C1) at ({2+cos((-pi/9) r)},{sin((-pi/9) r)});
\foreach \x in {A1,A2,B1,B2,C1,C2} {\fill (\x) circle (5.0pt);}
\draw [very thick] (A1) to[out=-30,in=-150] (-1.5,-2) to[out=30,in=-30] (B2);
\draw [very thick] (A2) to[out=30,in=150] (-1.5,2) to[out=-30,in=30] (B1);
\draw [very thick] (B1) to[out=-30,in=-150] (2.5,-1.5) to[out=30,in=-30] (C2);
\draw [very thick] (B2) to[out=30,in=150] (2.5,1.5) to[out=-30,in=30] (C1);
\draw [very thick] (A1) to[out=60,in=180] (1,2.5) to[out=0,in=60] (C2);
\draw [very thick] (A2) to[out=-60,in=180] (1,-2.5) to[out=0,in=-60] (C1);
\end{tikzpicture}
\caption{Pants thrackle drawing of a six-cycle.}
\label{figure:sixcycle}
\end{figure}
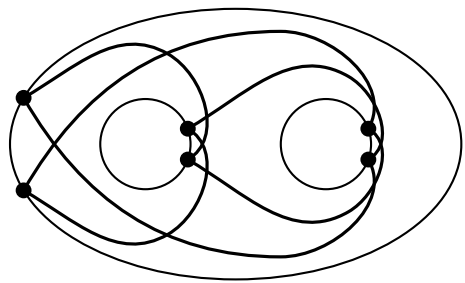

\medskip

We prove the following.

\begin{theorem} \label{t:pants}
Suppose a graph $G$ admits a pants thrackle drawing. Then
\begin{enumerate}[{\rm (a)}]
  \item \label{it:pantseven}
  any even cycle in $G$ is a six-cycle, and its drawing is Reidemeister equivalent to the one in Figure~\ref{figure:sixcycle};

  \item \label{it:pantsodd}
  if $G$ is an odd cycle, then the drawing can be obtained from a pants drawing of a three-cycle by a sequence of edge insertions;

  \item \label{it:pantsC}
  the number of edges of $G$ does not exceed the number of vertices.
\end{enumerate}
\end{theorem}

The procedure of edge insertion replaces an edge in a thrackle drawing by a three-path such that the resulting drawing is again a thrackle -- see Section~\ref{ss:ir} for details.

The ideas of the proofs are roughly as follows. There is a toolbox of operations one can do on a thrackled graph while preserving thrackleability and that have been used in the past literature on thrackles; these include edge insertion, edge removal and vertex splitting. We investigate how these operations interact with the more restrictive annular or pants conditions. One key observation (Lemma~\ref{lemma:triangle} below) is that, in order to preserve thrackleability, edge removal hinges on some empty triangle condition which blends well with the annular or the pants structure. This allows the study of
irreducible thrackles, which are those for which no edge removal is possible. We prove that irreducible thrackled cycles are either triangles, or, in the case of pants drawing, a six-cycle.


\section{Thrackle operations}
\label{s:pre}

\subsection{Edge insertion and edge removal}
\label{ss:ir}

The operation of edge insertion was introduced in \cite[Figure~14]{WOO71}; given a thrackle drawing, one replaces an edge by a three-path in such a way that the resulting drawing is again a thrackle. All the changes to the drawing are performed in a small neighbourhood of the edge, as shown in Figure~\ref{figure:insertion}.
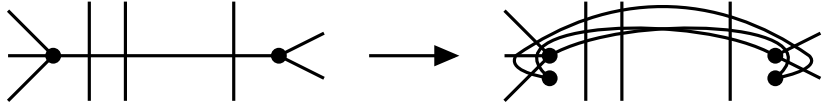
\begin{figure}[h]
\centering
\begin{tikzpicture}[scale=0.6,>=triangle 45]
\foreach \x in {0, 11}
{
\draw[very thick] (\x-1,1)-- (\x,0);
\draw[very thick] (\x-1,-1)-- (\x,0);
\draw[very thick] (\x-1,0)-- (\x,0);
\draw[very thick] (\x+6,0.5)-- (\x+5,0);
\draw[very thick] (\x+6,-0.5)-- (\x+5,0);
\fill (\x,0) circle (5.0pt); \fill (\x+5,0) circle (5.0pt);
\draw[very thick] (\x+0.8,-1) -- (\x+0.8,1.2);
\draw[very thick] (\x+1.6,-1) -- (\x+1.6,1.2);
\draw[very thick] (\x+4,-1) -- (\x+4,1.2);
}
\draw[->, very thick] (7,0) -- (9,0);
\draw[very thick] (0,0) -- (5,0);
\draw [very thick] (11,0) to[out=30,in=40] (16,-0.5);
\draw [very thick] (11,-0.5) to[out=140,in=150] (16,0);
\draw [very thick] (11,-0.5) to [out=170,in=-120] (10.25,0) to [out=35, in=145] (16.75,0) to [out=-45,in=10] (16,-0.5);
\fill (11,-0.5) circle (5.0pt); \fill (16,-0.5) circle (5.0pt);
\end{tikzpicture}
\caption{The edge insertion operation.}
\label{figure:insertion}
\end{figure}
Edge insertion on a given edge is not uniquely defined, even up to isotopy and Reidemeister moves, as we can choose one of two different orientations of the crossing of the first and the third edge of the three-path by which we replace the edge. We want to formalise and slightly modify the edge insertion procedure. Given an edge $e=uv$, on the first step, we remove from it a small segment $Q_1Q_2$ lying in the interior of $e$ and containing no crossings with other edges. On the second step, we slightly extend the segments $uQ_1$ and $Q_2v$ so that they cross (with one of two possible orientations), and then further extend each of them to cross other edges in such a way that the resulting drawing is again a thrackle. On the third step, we join the two endpoints of degree $1$ of the two edges obtained at the first step so that the resulting drawing is again a thrackle. We make two observations regarding this process of edge insertion. First, it may happen that we change the drawing not only in a small neighbourhood of $e$, but also ``far away" from it. Figure~\ref{figure:twosevens} shows two Reidemeister inequivalent thrackled seven-cycles obtained from the standard $5$-musquash by edge insertion. Note that the orientations of all the crossings in the two thrackles are the same (we note in passing, that up to isotopy and Reidemeister moves there exist only three thrackled seven-cycles: the two shown in Figure~\ref{figure:twosevens} and the standard $7$-musquash; we can prove that using the algorithm given in the end of Section~3 of \cite{MNajc}).
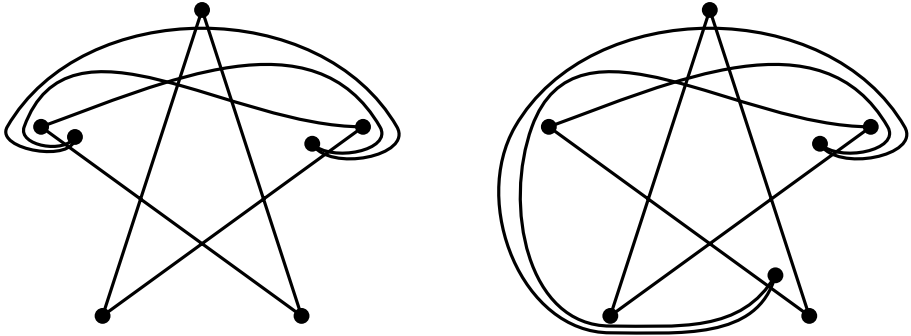
\begin{figure}[h]
\centering
\begin{tikzpicture}[scale=0.45]
\foreach \x in {0,15} {
\coordinate (A0) at (\x,5);
\coordinate (A1) at ({\x + 5*cos((2*pi/5+pi/2) r)},{5*sin((2*pi/5+pi/2) r)});
\coordinate (A2) at ({\x + 5*cos((2*pi*2/5+pi/2) r)},{5*sin((2*pi*2/5+pi/2) r)});
\coordinate (A3) at ({\x + 5*cos((2*pi*3/5+pi/2) r)},{5*sin((2*pi*3/5+pi/2) r)});
\coordinate (A4) at ({\x + 5*cos((2*pi*4/5+pi/2) r)},{5*sin((2*pi*4/5+pi/2) r)});
\coordinate (B1) at ($(A4)+(-1.5,-.5)$);
\foreach \y in {A0,A1,A2,A3,A4,B1} {\fill  (\y) circle (6.67pt);}
\draw[very thick] (A1) -- (A3) -- (A0) -- (A2) -- (A4);
\draw [very thick] (A1) to [out=20,in=120] ($(A4)+(0.5,0)$) to [out=-60,in=-45] (B1);
\ifthenelse{\x = 0}
    {\coordinate (B2) at ($(A1)+(1,-0.3)$); \fill(B2) circle(6.67pt);
    \draw [very thick] (A4) to[out=180,in=70] ($(A1)+(-0.5,0)$) to [out=-110,in=-120] (B2);
    \draw [very thick] (B2) to [out=-90,in=-120] ($(A1)+(-1,0)$) to [out=60,in=120] ($(A4)+(1,0)$) to [out=-60,in=-60] (B1)
    }
    {\coordinate (B2) at ($(A3)+(-1,1.2)$); \fill(B2) circle(6.67pt);
    \draw [very thick] (A4) to[out=180,in=70] ($(A1)+(-0.5,0)$) to [out=-110,in=180] ($(A2)+(0,-0.3)$) to [out=0,in=-120] (B2);
    \draw [very thick] (B2) to [out=-105,in=0] ($(A2)+(0,-0.5)$) to [out=180,in=-120] ($(A1)+(-1,0)$) to [out=60,in=120] ($(A4)+(1,0)$) to [out=-60,in=-60] (B1)
    }
  ;
}
\end{tikzpicture}
\caption{Two seven-cycles obtained by edge insertions on a five-cycle.}
\label{figure:twosevens}
\end{figure}

Our second observation is that edge insertion may not always be possible within the same class $T_d$. For example, in the proof of assertion~\eqref{it:pantseven} of Theorem~\ref{t:pants} in Section~\ref{s:pants}, it will be shown that no edge insertion on the pants thrackle drawing of the six-cycle shown in Figure~\ref{figure:sixcycle} produces a pants thrackle drawing. 

The operation of \emph{edge removal} is inverse to the edge insertion operation. Let $\T(G)$ be a thrackle drawing of a graph $G$ and let $v_1v_2v_3v_4$ be a three-path in $G$ such that $\deg v_2 = \deg v_3 = 2$. Let $Q = \T (v_1v_2) \cap \T (v_3v_4)$. Removing the edge $v_2v_3$, together with the segments $Qv_2$ and $Qv_3$ we obtain a drawing of a graph with a single edge $v_1v_4$ in place of the three-path $v_1v_2v_3v_4$ (Figure~\ref{figure:2}).
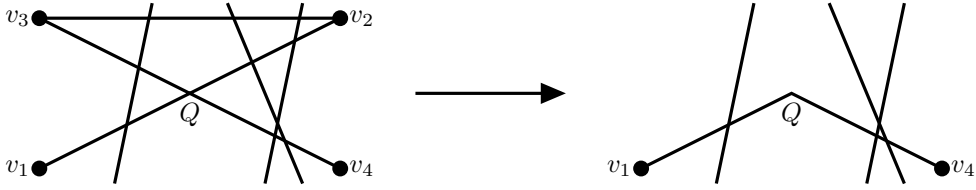
\begin{figure}[h]
\centering
\begin{tikzpicture}[>=triangle 45]
\node[coordinate] (v3) at (-2,2) [label=180:$v_3$] {}; \fill  (v3) circle (3pt);
\node[coordinate] (v2) at (2,2) [label=0:$v_2$] {}; \fill  (v2) circle (3pt);
\node[coordinate] (v1l) at (-2,0) [label=180:$v_1$] {}; \fill  (v1l) circle (3pt);
\node[coordinate] (v4l) at (2,0) [label=0:$v_4$] {}; \fill  (v4l) circle (3pt);
\node[coordinate] (Ql) at (0,1) [label=-90:$Q$] {};
\draw[very thick] (v1l) -- (v2) -- (v3) -- (v4l);
\draw[very thick] (-1,-0.2) -- (-0.5,2.2);
\draw[very thick] (1.5,-0.2) -- (0.5,2.2);
\draw[very thick] (1,-0.2) -- (1.5,2.2);
\draw[->, very thick] (3,1) -- (5,1);
\node[coordinate] (v1r) at (6,0) [label=180:$v_1$] {}; \fill  (v1r) circle (3pt);
\node[coordinate] (v4r) at (10,0) [label=0:$v_4$] {}; \fill  (v4r) circle (3pt);
\node[coordinate] (Qr) at (8,1) [label=-90:$Q$] {};
\draw[very thick] (v1r) -- (Qr) -- (v4r);
\draw[very thick] (7,-0.2) -- (7.5,2.2);
\draw[very thick] (9.5,-0.2) -- (8.5,2.2);
\draw[very thick] (9,-0.2) -- (9.5,2.2);
\end{tikzpicture}
\caption{The edge removal operation.}
\label{figure:2}
\end{figure}
Edge removal does not necessarily result in a thrackle drawing. Consider the triangular domain $\triangle$ bounded by the arcs $v_2v_3, \, Qv_2$ and $v_3Q$ and not containing the vertices $v_1$ and $v_4$ (if we consider the drawing on the plane, $\triangle$ can be unbounded). We have the following lemma.
\begin{lemma}[{\cite[Lemma~3]{GY2012}}]
\label{lemma:triangle}
Edge removal results in a thrackle drawing if and only if $\triangle$ contains no vertices of $\T(G)$.
\end{lemma}

Note that for a thrackle drawing of class $T_d$, the condition of Lemma~\ref{lemma:triangle} is satisfied if $\triangle$ contains none of the $d$ circles bounding the discs $D_k$. Given a thrackle drawing of class $T_d$ of an $n$-cycle, edge removal, if it is possible, produces a thrackle drawing of the same class $T_d$ of an $(n-2)$-cycle.

We call a thrackle drawing \emph{irreducible} if it admits no edge removals and \emph{reducible} otherwise.

To a path in a thrackle drawing of class $T_d$ we can associate a word $W$ in the alphabet $X=\{x_1, \dots, x_d\}$ in such a way that the $i$-th letter of $W$ is $x_k$ if the $i$-th vertex of the path lies on the boundary of the disc $D_k$. For a thrackled cycle, we consider the associated word $W$ to be a cyclic word. For a word $w$ and an integer $m$ we denote $w^m$ the word obtained by $m$ consecutive repetitions of $w$. We have the following simple observation.
\begin{lemma} \label{l:repeatsgen}
For a thrackle drawing of a graph $G$ of class $T_d$,
\begin{enumerate}[{\rm (a)}]
  \item \label{it:noaabbgen}
  For no two different $i, j = 1, \dots, d$, may a thrackle drawing of class $T_d$ contain two edges with the words $x_i^2$ and $x_j^2$.

  \item \label{it:noaaagen}
  Suppose that for some $i = 1, \dots, d$, a thrackle drawing of class $T_d$ contains a two-path with the word $x_i^3$ the first two vertices of which have degree $2$. Then the drawing is reducible.
\end{enumerate} 
\end{lemma}
\begin{proof}
\eqref{it:noaabbgen} is obvious, as otherwise the thrackle condition will be violated by the corresponding two edges.

\eqref{it:noaaagen} The complement of the two-path in $S^2 \setminus (\cup_{k=1}^d \overline{D_k})$ is the union of three domains, exactly one of which has the two-path on its boundary. That domain can contain no other vertices of the thrackle inside it or on its boundary, as otherwise the thrackle condition is violated. But then by Lemma~\ref{lemma:triangle}, edge removal can be performed on the three-path which is the union of the given two-path and the edge of the graph incident to its first vertex.
\end{proof}

\subsection{Reidemeister moves}
\label{ss:R}

A \emph{Reidemeister move} can be performed on a triple of pairwise non-adjacent edges of a thrackle drawing if the open triangular domain bounded by the segments on each of the edges between the crossings with the other two contains no points of the drawing -- see Figure~\ref{figure:Reid}.
\begin{figure}[h]
\centering
\begin{tikzpicture}[scale=0.6,>=triangle 45]
\foreach \x in {0,8} {
\draw[very thick] ({\x+2*cos(pi/3 r)},{2*sin(pi/3 r)}) -- ({\x+2*cos(4*pi/3 r)},{2*sin(4*pi/3 r)});
\draw[very thick] ({\x+2*cos(2*pi/3 r)},{2*sin(2*pi/3 r)}) -- ({\x+2*cos(5*pi/3 r)},{2*sin(5*pi/3 r)});
\ifthenelse{\x = 0}
    {\draw[very thick] (\x-2,0) to [out=20,in=160] (\x+2,0)}
    {\draw[very thick] (\x-2,0) to [out=-20,in=-160] (\x+2,0)}
  ;
}
\draw[->, very thick] (3,0) -- (5,0);
\end{tikzpicture}
\caption{A Reidemeister move.}
\label{figure:Reid}
\end{figure}
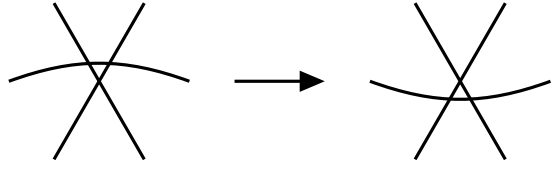
We say that two thrackle drawings are \emph{Reidemeister equivalent} if one can be obtained from the other by a finite sequence of Reidemeister moves.

Suppose that two thrackles $\T_1(G)$ and $\T_2(G)$ can be obtained from one another by a Reidemeister move on a triple of edges $e_i, e_j, e_k$. From Lemma~\ref{lemma:triangle} it follows that if $\T_1(G)$ admits edge removal on a three-path not containing these three edges, then $\T_2(G)$ also does; moreover, after edge removals the resulting two thrackles can again be obtained from one another by the same Reidemeister move. However, adding an edge to $\T_1(G)$ may result in a thrackle which is not Reidemeister equivalent to any thrackle obtained from $\T_2(G)$ by adding an edge, as the added edge may end at a vertex inside the triangular domain $\triangle_{ijk}$ bounded by $e_i, e_j, e_k$. The same is true for edge insertion on $\T_1(G)$.

Now suppose that $\T_1(G)$ and $\T_2(G)$ belong to a class $T_d$. The domains $\triangle_{ijk}$ in both $\T_1(G)$ and $\T_2(G)$ contain no vertices. If we additionally require that they contain no ``inessential" discs $D_l$, those having no vertices on their boundaries, then the edge added to $\T_1(G)$ cannot end in $\triangle_{ijk}$ and so we can add a corresponding edge to $\T_2(G)$ such that the resulting two thrackles are again Reidemeister equivalent. 

\subsection{Forbidden configurations}
\label{ss:forbidden}

A graph having more edges than vertices always contains one of the following subgraphs: a theta-graph (two vertices joined by three disjoint paths), a dumbbell (two disjoint cycles with a path joining a vertex of one cycle to a vertex of another), or a figure-$8$ graph (two cycles sharing a vertex). To prove Conway's Thrackle Conjecture it is therefore sufficient to show that none of these three graphs admits a thrackle drawing. Repeatedly using the vertex-splitting operation \cite[Figure~1(a)]{MNajc} one can show that the existence of a counterexample of any of these three types implies the existence of a counterexample of the other two types.
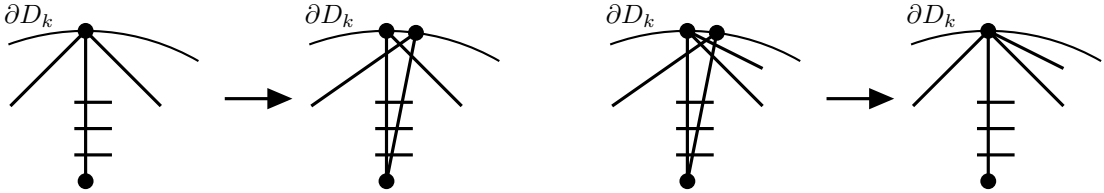
\begin{figure}[h]
\centering
\begin{tikzpicture}[scale=0.5,>=triangle 45]
\foreach \x in {0,8} {
\draw[thick] (\x,0) arc (60:110:6);
\coordinate (A1) at ({\x-6*cos(pi/3 r)},{6-6*sin(pi/3 r)});
\coordinate (A2) at ($(A1)-(0,4)$);
\foreach \y in {A1,A2} {\fill (\y) circle (6.0pt);}
\draw (\x-4.5,1.2) node {$\partial D_k$};
\draw[very thick] (A1)-- (A2);
\draw[very thick] ($(A1)+(2,-2)$) -- (A1);
\foreach \z in {1,2,3} {\draw[very thick] ($(A2)+(-0.3,\z*0.7)$) -- ($(A2)+(0.7,\z*0.7)$);}
\ifthenelse{\x = 0}
    {\draw[very thick] ($(A1)+(-2,-2)$) -- (A1)}
    {\coordinate (A3) at ($(A1)+(0,-6)+({6*cos(11*pi/24 r)},{6*sin(11*pi/24 r)})$); \fill(A3) circle(6pt);
    \draw [very thick] (A2) -- (A3) -- ($(A1)+(-2,-2)$)}
  ;
}
\draw[->, very thick] (0.7,-1) -- (2.5,-1);
\foreach \x in {16,24} {
\draw[thick] (\x,0) arc (60:110:6);
\coordinate (A1) at ({\x-6*cos(pi/3 r)},{6-6*sin(pi/3 r)});
\coordinate (A2) at ($(A1)-(0,4)$);
\foreach \y in {A1,A2} {\fill (\y) circle (6.0pt);}
\draw (\x-4.5,1.2) node {$\partial D_k$};
\draw[very thick] (A1)-- (A2);
\draw[very thick] ($(A1)+(2,-2)$) -- (A1);
\draw[very thick] ($(A1)+(2,-1)$) -- (A1);
\foreach \z in {1,2,3} {\draw[very thick] ($(A2)+(-0.3,\z*0.7)$) -- ($(A2)+(0.7,\z*0.7)$);}
\ifthenelse{\x = 24}
    {\draw[very thick] ($(A1)+(-2,-2)$) -- (A1)}
    {\coordinate (A3) at ($(A1)+(0,-6)+({6*cos(11*pi/24 r)},{6*sin(11*pi/24 r)})$); \fill(A3) circle(6pt);
    \draw [very thick] (A2) -- (A3) -- ($(A1)+(-2,-2)$)}
  ;
}
\draw[->, very thick] (16.7,-1) -- (18.5,-1);
\end{tikzpicture}
\caption{Splitting a vertex of degree $3$ and a vertex of degree $4$.}
\label{figure:splitting}
\end{figure}
However, this may not be true for thrackle drawings of class $T_d$, as the required vertex-splitting operation on a vertex of degree $3$ may not be permitted within the class $T_d$. The problem is that in order to remain within  the class $T_d$, vertex-splitting  on a vertex of degree 3 may only be performed by doubling the ``middle" edge (as on the left in Figure~\ref{figure:splitting}) and this is too restrictive; for example, starting with a dumbbell, vertex-splitting within  the class $T_d$ might only \emph{increase} the length of the dumbbell handle. So one might not be able to reduce a dumbbell to a figure-$8$ graph. Nevertheless, if we are given a thrackle drawing of class $T_d$ of a figure-$8$ graph, we can always perform the vertex-splitting operation on the vertex of degree $4$ to obtain a thrackle drawing of the same class $T_d$ of a dumbbell, as on the right in Figure~\ref{figure:splitting}. This gives the following lemma.
\begin{lemma} \label{l:TCTd}
To prove Conway's Thrackle Conjecture for thrackle drawings in a class $T_d$ it is sufficient to prove that no dumbbell and no theta-graph admit a thrackle drawing of class $T_d$. In both cases, the corresponding graph contains an even cycle.
\end{lemma}
The second assertion is clear for a theta-graph, and for a dumbbell, follows from the fact that a thracklable graph contains no two vertex-disjoint odd cycles \cite[Lemma~2.1]{LPS97}.

\section{Annular thrackles}
\label{s:ann}

In this section, we prove Theorem~\ref{t:ann}. We can assume that the thrackle drawing lies in the closed annulus bounded by two concentric circles on the plane, the outer circle $A$ and the inner circle $B$; the vertices lie in $A \cup B$, and the rest of the drawing, in the open annulus. As in Section~\ref{ss:ir} we can associate to a path within a thrackle a word in the alphabet $\{a,b\}$, where the letter $a$ (respectively $b$) corresponds to a vertex lying on $A$ (respectively on $B$). To an annular thrackle drawing of an $n$-cycle there corresponds a word $W$ defined up to cyclic permutation and reversing.

The following lemma and the fact that edge removal decreases the length of a cycle by $2$ imply assertion~\eqref{it:annodd}.

\begin{lemma} \label{l:3cycle}
If an $n$-cycle admits an irreducible annular thrackle drawing, then $n = 3$.
\end{lemma}
\begin{proof}
By Lemma~\ref{l:repeatsgen}\eqref{it:noaabbgen} we can assume that $W$ contains no two consecutive $b$'s. If $W$ contains no letters $b$ at all, then the thrackle is outerplanar and the assertion of the lemma follows from Theorem~\ref{t:outer} \eqref{it:outRei}. Assuming that $W$ contains at least one $b$ we get that $W$ contains a sequence $aba$. Suppose $n>3$; then $n \ge 5$, as no $4$-cycle admits a thrackle drawing on the plane. Consider the next letter in $W$. Up to isotopy, there are three possible ways of adding an extra edge. As the reader may verify, two of them produce a reducible thrackle by Lemma~\ref{lemma:triangle}. The third one is shown in the middle in Figure~\ref{figure:aba}.
\begin{figure}[h]
\centering
\begin{tikzpicture}[scale=0.5,>=triangle 45]
\foreach \x in {0,10,20} {
\draw[thick] (\x,0) circle (4); \draw[thick] (\x,0) circle (1); 
\coordinate (B) at (\x,1);
\coordinate (A2) at ({\x+4*cos((pi/6) r)},{4*sin((pi/6) r)});
\coordinate (A1) at ({\x+4*cos((5*pi/6) r)},{4*sin((5*pi/6) r)}) circle (5.45pt);
\fill (A1) circle (6pt);
\fill (A2) circle (6pt);
\fill (B) circle (6pt);
\draw [very thick] (A1)--(B)--(A2);
\ifthenelse{\x = 0}{}
    {
    \coordinate (A3) at ({\x+4*cos((2*pi/3) r)},{4*sin((2*pi/3) r)});
    \fill (A3) circle (6pt);
    \draw [very thick] (A2) to[out=-90,in=0] ({\x-1},-2) to [out=180,in=-90] (A3);
    \ifthenelse{\x = 10}{}
    {
        \coordinate (A4) at ({\x+4*cos((pi/3) r)},{4*sin((pi/3) r)});
        \fill (A4) circle (6pt);
        \draw [very thick] (A4) to[out=-80,in=0] (\x-0.5,-1.5) to[out=180,in=-90] (A3);
    }
}
  ;
}
\draw[->, very thick] (4.5,0) -- (5.5,0); \draw[->, very thick] (14.5,0) -- (15.5,0);
\end{tikzpicture}
\caption{Adding the third and the fourth edge.}
\label{figure:aba}
\end{figure}
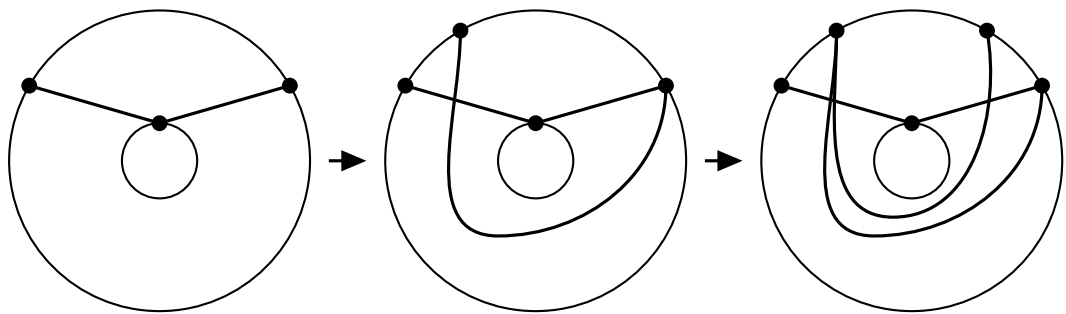
But then there is only one way to add the next edge, as on the right in Figure~\ref{figure:aba} and the resulting thrackle drawing is reducible.
\end{proof}

By Lemma~\ref{l:TCTd}, if in the class of annular thrackles there exists a counterexample to Conway's Thrackle Conjecture, then there exists such a counterexample whose underlying graph contains an even cycle. So assertion~\eqref{it:annC} follows from assertion~\eqref{it:annodd}.

We now prove assertion~\eqref{it:annout}. Suppose a cycle $c$ of an odd length $n$ admits an annular thrackle drawing. We can assume that the corresponding word $W$ contains at least one $b$ and does not contain $b^2$.
\begin{lemma} \label{l:alt}
Up to cyclic permutation, $W=a^{2p}(ba)^rb$, for some $p \ge 1, \, r \ge 0$.
\end{lemma}
\begin{proof}
As $n$ is odd, $W$ contains a subword $a^2$. Let $a^k, \; k \ge 2$, be a maximal by inclusion string of consecutive $a$'s. If $k=n-1$, we are done. Otherwise, up to cyclic permutation, $W=a^kb w b$ for some word $w$. Consider the edge $e$ defined by the last pair $aa$ in $a^k$. Let $\gamma$ be the arc of $A$ joining the endpoints of $e$ such that the domain bounded by $e \cup \gamma$ does not contain $B$. Every edge of the thrackle not sharing a common vertex with $e$ crosses it, so every second vertex counting from the last $a$ in $a^k$ lies in the interior of $\gamma$. It follows that $W=a^kbay_1ay_2a \dots y_q a b$, where $y_i \in \{a, b\}$, and so $k$ is necessarily even. By the same reasoning, any maximal sequence of more than one consecutive $a$'s in $W$ is even. But then $y_i=b$, for all $i=1, \dots, q$, as otherwise $W$ would contain a maximal sequence of consecutive $a$'s of an odd length greater than one.
\end{proof}

To prove assertion~\eqref{it:annout} we show any annular thrackled cycle is alternating; then the claim follows from the fact that alternating thrackles are outerplanar, as was proved in \cite[Theorem~2]{GY2012}. Recall that a thrackled cycle is called \emph{alternating} if for every edge $e$ and every two-path $fg$ vertex-disjoint from $e$, the crossings of $e$ by $f$ and $g$ have opposite orientations.

Suppose $c$ is a cycle of the shortest possible length which admits a non-alternating annular thrackle drawing $\T(c)$; the length of $c$ must be at least $7$. An easy inspection shows that any edge vertex-disjoint with a two-path $aba$ (or $bab$) crosses its edges with opposite orientations. The same is true for a two-path $a^3$. It remains to show that any edge vertex-disjoint with a two-path $aab$ also crosses the edges of that two-path with opposite orientations. Up to isotopy, the only drawing for which this is not true is the one shown in Figure~\ref{figure:annaab}.
\begin{figure}[h]
\centering
\begin{tikzpicture}[scale=0.5,>=triangle 45]
\draw[thick] (0,0) circle (4); \draw[thick] (0,0) circle (1);
\node[coordinate] (A1) at ({4*cos((7*pi/9) r)},{4*sin((7*pi/9) r)}) [label=90:$a_1$] {};
\node[coordinate] (A2) at ({4*cos((11*pi/9) r)},{4*sin((11*pi/9) r)}) [label=-90:$a_2$] {};
\node[coordinate] (B1) at (-1,0) [label=0:$b_1$] {};
\node[coordinate] (Bp) at (0,1) [label=-90:$b'$] {};
\node[coordinate] (Ap) at ({-sqrt(15)},1) [label=180:$a'$] {};
\foreach \x in {A1,A2,Ap,Bp,B1} {\fill (\x) circle (6pt);}
\draw [very thick] (A1)--(A2);
\draw [very thick] (Ap)--(Bp);
\draw [very thick] (A2) to [out=0,in=-90] (2.5,0) to [out=90,in=0] (0,2) to [out=180,in=90] (B1);
\end{tikzpicture}
\caption{A non-alternating crossing.}
\label{figure:annaab}
\end{figure}
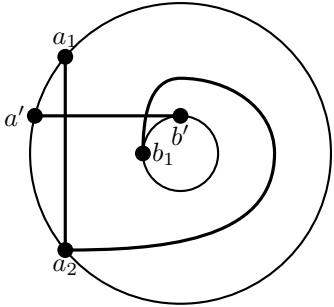
Note that the edge which violates the alternating condition necessarily joins an $a$-vertex and a $b$-vertex. We claim that such a drawing cannot be a part of $\T(c)$. To see that, we consider possible drawings of the four-path in $c$ which extends the path $a_1a_2b_1$.
The vertex following $b_1$ must be an $a$-vertex (call it $a_3$); there are two possible cases: $a_3 = a'$ and $a_3 \ne a'$. In the first case, up to isotopy, we get the drawing on the left in Figure~\ref{figure:annaaba1}, and then there is only one possible way to attach an edge at $a_1$, as shown on the right in Figure~\ref{figure:annaaba1}. But then performing edge removal on $a_1a_2$ we get a shorter non-alternating annular thrackled cycle, a contradiction.
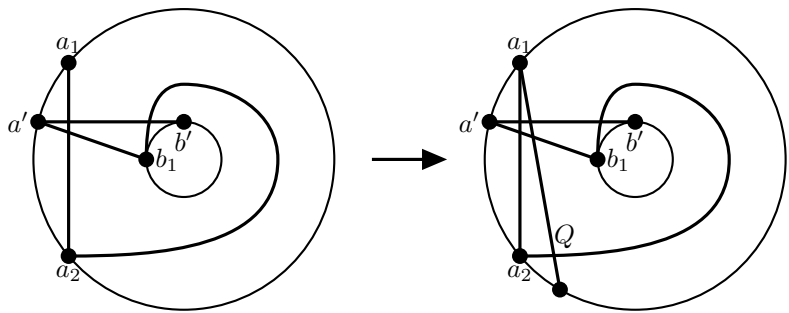
\begin{figure}[h]
\centering
\begin{tikzpicture}[scale=0.5,>=triangle 45]
\foreach \y in {0,12} {
\draw[thick] (\y,0) circle (4); \draw[thick] (\y,0) circle (1);
\node[coordinate] (A1) at ({\y+4*cos((7*pi/9) r)},{4*sin((7*pi/9) r)}) [label=90:$a_1$] {};
\node[coordinate] (A2) at ({\y+4*cos((11*pi/9) r)},{4*sin((11*pi/9) r)}) [label=-90:$a_2$] {};
\node[coordinate] (B1) at (\y-1,0) [label=0:$b_1$] {};
\node[coordinate] (Bp) at (\y,1) [label=-90:$b'$] {};
\node[coordinate] (Ap) at ({\y-sqrt(15)},1) [label=180:$a'$] {};
\foreach \x in {A1,A2,Ap,Bp,B1} {\fill (\x) circle (6pt);}
\draw [very thick] (A1)--(A2);
\draw [very thick] (Ap)--(Bp);
\draw [very thick] (A2) to [out=0,in=-90] (\y+2.5,0) to [out=90,in=0] (\y,2) to [out=180,in=90] (B1);
\draw [very thick] (Ap)--(B1);
\ifthenelse{\y = 12}
    {\coordinate (AA) at (\y-2,{-sqrt(12)}); \fill(AA) circle(6pt);
    \draw [very thick] (A1)--(AA);
    \node[coordinate] (Q) at (\y-2.4,-2.6) [label=45:$Q$] {};
    }
  ;
}
\draw[->, very thick] (5,0) -- (7,0);
\end{tikzpicture}
\caption{Path $a_1a_2b_1a_3, \; a_3 = a'$.}
\label{figure:annaaba1}
\end{figure}
Now suppose $a_3 \ne a'$. We have two cases for adding the edge $b_1a_3$, and then by Lemma~\ref{l:alt}, the letter after $a_3$ must be a $b$. In the first case, up to isotopy and a Reidemeister move, we get the drawing on the left in Figure~\ref{figure:annaaba2}, and then we can attach the edge joining $a_3$ to a $b$-vertex uniquely, up to isotopy and a Reidemeister move, as on the right in Figure~\ref{figure:annaaba2}.
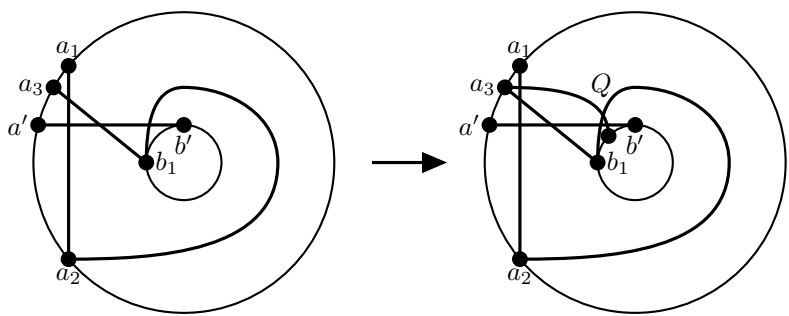
\begin{figure}[h]
\centering
\begin{tikzpicture}[scale=0.5,>=triangle 45]
\foreach \y in {0,12} {
\draw[thick] (\y,0) circle (4); \draw[thick] (\y,0) circle (1);
\node[coordinate] (A1) at ({\y+4*cos((7*pi/9) r)},{4*sin((7*pi/9) r)}) [label=90:$a_1$] {};
\node[coordinate] (A2) at ({\y+4*cos((11*pi/9) r)},{4*sin((11*pi/9) r)}) [label=-90:$a_2$] {};
\node[coordinate] (B1) at (\y-1,0) [label=0:$b_1$] {};
\node[coordinate] (Bp) at (\y,1) [label=-90:$b'$] {};
\node[coordinate] (Ap) at ({\y-sqrt(15)},1) [label=180:$a'$] {};
\node[coordinate] (A3) at ({\y-sqrt(12)},2) [label=180:$a_3$] {};
\foreach \x in {A1,A2,Ap,Bp,B1,A3} {\fill (\x) circle (6pt);}
\draw [very thick] (A1)--(A2);
\draw [very thick] (Ap)--(Bp);
\draw [very thick] (A2) to [out=0,in=-90] (\y+2.5,0) to [out=90,in=0] (\y,2) to [out=180,in=90] (B1);
\draw [very thick] (A3)--(B1);
\ifthenelse{\y = 12}
    {\coordinate (B2) at ({\y-sqrt(2)/2},{sqrt(2)/2}); \fill(B2) circle(6pt);
    \draw [very thick] (A3) to [out=0,in=90] (B2);
    \node[coordinate] (Q) at (\y-0.9,1.5) [label=90:$Q$] {};
    }
  ;
}
\draw[->, very thick] (5,0) -- (7,0);
\end{tikzpicture}
\caption{Path $a_1a_2b_1a_3, \; a_3 \ne a'$, case 1.}
\label{figure:annaaba2}
\end{figure}
Again, performing edge removal on $b_1a_3$ we get a shorter non-alternating annular thrackled cycle.
The second possibility of attaching the edge $b_1a_3, \; a_3 \ne a'$, to the drawing in Figure~\ref{figure:annaab} is the one shown on the left in Figure~\ref{figure:annaaba3}, up to isotopy. Then the edge joining $a_3$ to the next $b$-vertex can be also added uniquely, up to isotopy, as on the right in Figure~\ref{figure:annaaba3}, and yet again, edge removal on $b_1a_3$ results in a shorter non-alternating annular thrackled cycle. This completes the proof of Theorem~\ref{t:ann}.
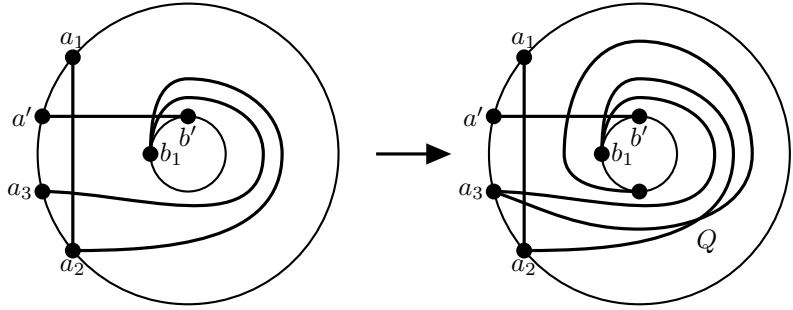
\begin{figure}[h]
\centering
\begin{tikzpicture}[scale=0.5,>=triangle 45]
\foreach \y in {0,12} {
\draw[thick] (\y,0) circle (4); \draw[thick] (\y,0) circle (1);
\node[coordinate] (A1) at ({\y+4*cos((7*pi/9) r)},{4*sin((7*pi/9) r)}) [label=90:$a_1$] {};
\node[coordinate] (A2) at ({\y+4*cos((11*pi/9) r)},{4*sin((11*pi/9) r)}) [label=-90:$a_2$] {};
\node[coordinate] (B1) at (\y-1,0) [label=0:$b_1$] {};
\node[coordinate] (Bp) at (\y,1) [label=-90:$b'$] {};
\node[coordinate] (Ap) at ({\y-sqrt(15)},1) [label=180:$a'$] {};
\node[coordinate] (A3) at ({\y-sqrt(15)},-1) [label=180:$a_3$] {};
\foreach \x in {A1,A2,Ap,Bp,B1,A3} {\fill (\x) circle (6pt);}
\draw [very thick] (A1)--(A2);
\draw [very thick] (Ap)--(Bp);
\draw [very thick] (A2) to [out=0,in=-90] (\y+2.5,0) to [out=90,in=0] (\y,2) to [out=180,in=90] (B1);
\draw [very thick] (A3) to [out=0,in=-90] (\y+2,0) to [out=90,in=0] (\y,1.5) to [out=180,in=90] (B1);
\ifthenelse{\y = 12}
    {\coordinate (B2) at (\y,-1); \fill(B2) circle(6pt);
    \draw [very thick] (A3) to [out=-20,in=180] (\y,-2) to [out=0,in=-90] (\y+3,0) to [out=90,in=0] (\y,3) to [out=180,in=90] (\y-2,0) to [out=-90,in=180] (B2);
    \node[coordinate] (Q) at (\y+1.8,-2.9) [label=90:$Q$] {};
    }
  ;
}
\draw[->, very thick] (5,0) -- (7,0);
\end{tikzpicture}
\caption{Path $a_1a_2b_1a_3, \; a_3 \ne a'$, case 2.}
\label{figure:annaaba3}
\end{figure}

\section{Pants thrackles}
\label{s:pants}

In this section, we prove Theorem~\ref{t:pants}. We represent the pair of pants domain $P$ whose closure contains the drawing as the interior of an ellipse, with two disjoint closed discs removed. To a path in a pants thrackle drawing we associate a word in the alphabet $\{a, b, c\}$, where $a$ corresponds to the vertices on the ellipse, and $b$ and $c$, to the vertices on the circles bounding the discs (e.g., as in Figure~\ref{figure:cabac}).

We start with the following proposition which implies assertion~\eqref{it:pantsodd} of Theorem~\ref{t:pants} and will also be used in the proof of assertion~\eqref{it:pantseven}.
{
\begin{proposition*} 
If a cycle $C$ admits an irreducible pants thrackle drawing, then $C$ is either a three-cycle or a six-cycle, and in the latter case, the drawing is Reidemeister equivalent to the one in Figure~\ref{figure:sixcycle}.
\end{proposition*}
\begin{proof}
Let $W$ be the (cyclic) word corresponding to an irreducible pants thrackle drawing of a cycle $C$. The following lemma can be compared to Lemma~\ref{l:alt}.
{
\begin{lemma} \label{l:repeats}
If $W$ contains $a^2$, then one of the two domains of the complement of the corresponding edge in $P$ is a disc, the cycle $C$ is odd, and $W=aay_1ay_2 \dots y_{m-1}ay_m$, where $y_i \in \{b, c\}$ for $i=1, \dots, m$.
\end{lemma}
\begin{proof}
Suppose no domain of the complement of an edge $aa$ is a disc. By Lemma~\ref{l:repeatsgen}\eqref{it:noaaagen}, neither the letter which precedes $a^2$ in $W$, nor the next letter after $a^2$ is $a$, and for the corresponding edges to cross, those two letters must be the same, say $b$. If the corresponding three-path $baab$ is irreducible, it has to be isotopic to the path on the left in Figure~\ref{figure:aanonzero}. But then there is a unique, up to isotopy, way to add to the path the starting segment of the next edge, and it produces a reducible three-path, as on the right in Figure~\ref{figure:aanonzero}.
\begin{figure}[h]
\centering
\begin{tikzpicture}[scale=0.6,>=triangle 45]
\foreach \y in {0,13} {
\draw[thick]  (\y,0) ellipse (5 and 3); \draw[thick]  (\y-2,0) circle (1); \draw[thick]  (\y+2,0) circle (1);
\node[coordinate] (A1) at (\y,3) [label=-45:$a$] {};
\node[coordinate] (A2) at (\y,-3) [label=45:$a$] {};
\node[coordinate] (B1) at (\y-2,-1) [label=90:$b$] {};
\node[coordinate] (B2) at (\y-1,0) [label=180:$b$] {};
\foreach \x in {A1,A2,B1,B2} {\fill (\x) circle (5.0pt);}
\draw [very thick] (A1) to (A2);
\draw [very thick] (A2) to (B1);
\draw [very thick] (A1) to[out=-135,in=90] (\y-4,0) to [out=-90,in=180] (\y-2,-2) to [out=0,in=-45] (B2);
}
\draw [very thick] (B1) to (13.5,-1);
\draw [very thick,dashed] (13.5,-1)--(14.5,-1);
\draw[->, very thick] (5.5,0) -- (7.5,0);
\end{tikzpicture}
\caption{Adding the fourth edge produces a reducible path.}
\label{figure:aanonzero}
\end{figure}
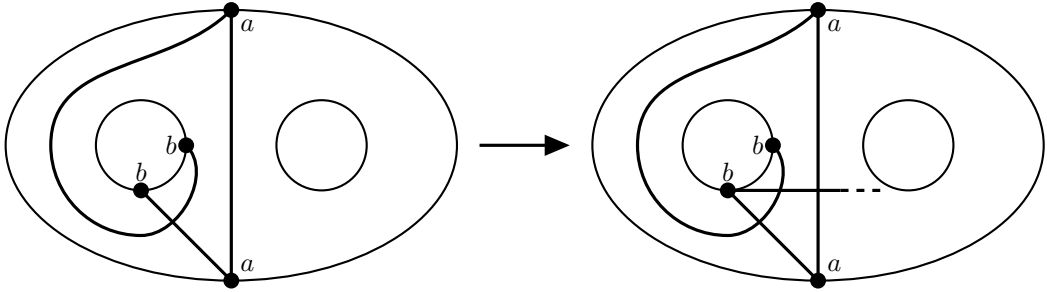
It follows that if $W$ contains $a^2$, then one of the two domains of the complement of the corresponding edge in $P$ is a disc. But then by the thrackle condition, every second vertex counting from the second $a$ in $aa$ is again $a$, so $W=aay_1ay_2 \dots y_{m-1}ay_m$, for some $y_i \in \{a, b, c\}$. In particular, $C$ is an odd cycle and furthermore, none of the $y_i$ can be equal to $a$ by Lemma~\ref{l:repeatsgen}\eqref{it:noaaagen}.
\end{proof}
}

\begin{lemma} \label{l:caba}
Suppose the word $W$ contains no subwords $bb$ and $cc$. Then it contains no subwords $caba$ or $baca$.
\end{lemma}
\begin{proof}
Arguing by contradiction (and renaming the letters if necessary) suppose that $W$ contains the subword $caba$. The only irreducible three-path corresponding to that subword, up to isotopy, is shown on the left in Figure~\ref{figure:cabac}. Suppose that the next letter in $W$ is not $a$. Then the only irreducible four-path extending $caba$, up to isotopy, is the one shown on the right in Figure~\ref{figure:cabac}.
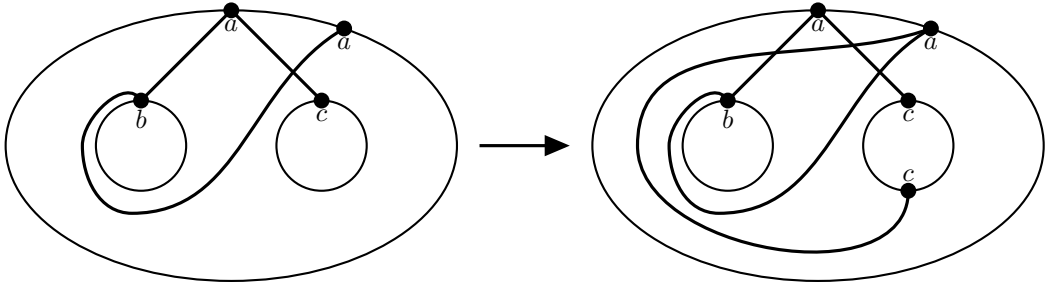
\begin{figure}[h]
\centering
\begin{tikzpicture}[scale=0.6,>=triangle 45]
\foreach \y in {0,13} {
\draw[thick]  (\y,0) ellipse (5 and 3); \draw[thick]  (\y-2,0) circle (1); \draw[thick]  (\y+2,0) circle (1);
\node[coordinate] (A1) at (\y,3) [label=-90:$a$] {};
\node[coordinate] (A2) at ({\y+5*cos((pi/3) r)},{3*sin((pi/3) r)}) [label=-90:$a$] {};
\node[coordinate] (B1) at (\y-2,1) [label=-90:$b$] {};
\node[coordinate] (C1) at (\y+2,1) [label=-90:$c$] {};
\foreach \x in {A1,A2,B1,C1} {\fill (\x) circle (5.0pt);}
\draw [very thick] (C1) to (A1) to (B1);
\draw [very thick] (B1) to [out=120,in=90] (\y-3.3,0) to [out=-90,in=180] (\y-2.2,-1.5) to [out=0,in=-150] (A2);
}
\node[coordinate] (C2) at (15,-1) [label=90:$c$] {}; \fill (C2) circle (5.0pt);
\draw [very thick] (A2) to [out=-160,in=90] (9,0) to [out=-90,in=-90] (C2);
\draw[->, very thick] (5.5,0) -- (7.5,0);
\end{tikzpicture}
\caption{The irreducible path $caba$ and the next edge ending not in $a$.}
\label{figure:cabac}
\end{figure}
If $C$ is of length five, then $W=cabac$ which contradicts the fact that the (cyclic) word $W$ does not contain a subword $cc$. Otherwise, there are only three possible ways, up to isotopy and a Reidemeister move, to add another edge starting at the last added vertex $c$ in such a way that the resulting drawing is a thrackled path. But one of them results in a reducible drawing, and the other two end in $c$ contradicting the fact that $W$ does not contain a subword $cc$.

It follows that the letter following $caba$ in $W$ must be an $a$, so we get a subword $cabaa$. If the length of the cycle $C$ is greater than $5$, then by Lemma~\ref{l:repeats}, the letter which precedes $c$ must be $a$, so $W$ contains the subword $acabaa$. But then the above argument applied to the subword $acab$ (if we reverse the direction of $C$ and swap $b$ and $c$) implies that the letter which precedes the starting $a$ is another $a$, so that $W$ contains the subword $aacabaa$, giving a contradiction with Lemma~\ref{l:repeats}.

If $C$ is of length five, then $W=cabaa$ and the resulting drawing is reducible by Lemma~\ref{lemma:triangle}, as there is just a single $b$ in $W$, and so the triangular domain corresponding to the three-path $caba$ on the left in Figure~\ref{figure:cabac} contains no other vertices of the thrackle.
\end{proof}

Now if $W$ contains the subword $aa$, then by Lemma~\ref{l:repeats} and Lemma~\ref{l:caba}, the word $W$ may contain only one of the letters $b$ or $c$. Then the drawing is annular, and hence by Lemma~\ref{l:3cycle} is reducible unless $C$ is a three-cycle. Suppose $W$ contains no letter repetitions. Then Lemma~\ref{l:caba} applies to any subword $xyzy$ such that $\{x, y, z\} = \{a, b, c\}$. Furthermore, up to renaming the letters we can assume that $W$ starts with $ab$. If $W$ contains no subword $abc$, then $W=(ab)^m$, and so the drawing is annular. We can therefore assume that $W$ contains a subword $abc$. Then the following letter cannot be any of $b$ or $c$, so it must be an $a$. Repeating this argument we obtain that $W=(abc)^m$.

We now modify the word $W$ by attaching to every letter a subscript plus (respectively minus) if the tangent vector to the drawing in the direction of the cycle $C$ makes a positive (respectively negative) turn at the corresponding vertex; in other words, the subscript is a plus (respectively a minus) if the path turns left (respectively right) at the vertex. We will occasionally omit the subscript when it is unknown or unimportant.

Note that if the length of $C$ is greater than $3$, then no two consecutive subscripts in the word $W=(abc)^m$ can be the same. Indeed, assume that $W$ contains a subword $ab_+c_+a$. Then the corresponding irreducible three-path is unique up to isotopy, as shown on the left in Figure~\ref{figure:ab+c+a}, and the only possible way to attach an edge $ab$ results in a reducible drawing, as on the right in Figure~\ref{figure:ab+c+a}. By reflection, a similar comment applies to subwords of the form $ab_-c_-a$.
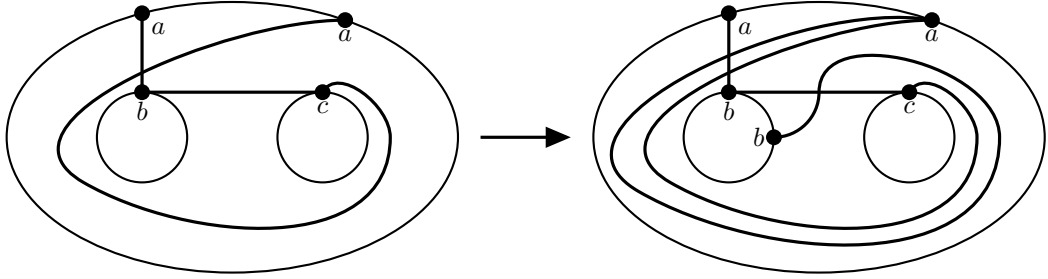
\begin{figure}[h]
\centering
\begin{tikzpicture}[scale=0.6,>=triangle 45]
\foreach \y in {0,13} {
\draw[thick]  (\y,0) ellipse (5 and 3); \draw[thick]  (\y-2,0) circle (1); \draw[thick]  (\y+2,0) circle (1);
\node[coordinate] (A1) at (\y-2,{3*sqrt(21)/5}) [label=-45:$a$] {};
\node[coordinate] (A2) at ({\y+5*cos((pi/3) r)},{3*sin((pi/3) r)}) [label=-90:$a$] {};
\node[coordinate] (B1) at (\y-2,1) [label=-90:$b$] {};
\node[coordinate] (C1) at (\y+2,1) [label=-90:$c$] {};
\foreach \x in {A1,A2,B1,C1} {\fill (\x) circle (5.0pt);}
\draw [very thick] (A1) to (B1) to (C1);
\draw [very thick] (C1) to [out=60,in=90] (\y+3.5,0) to [out=-90,in=-30] (\y-3.3,-1) to [out=150,in=-180] (A2);
}
\node[coordinate] (B2) at (12,0) [label=180:$b$] {}; \fill (B2) circle (5.0pt);
\draw [very thick] (A2) to [out=170,in=150] (9,-1.2) to [out=-30,in=-90] (17,0) to [out=90,in=90] (13,1) to [out=-90,in=0] (B2);
\draw[->, very thick] (5.5,0) -- (7.5,0);
\end{tikzpicture}
\caption{The path $ab_+c_+a$ extends to a reducible drawing.}
\label{figure:ab+c+a}
\end{figure}
It follows that the subscripts in $W$ alternate and in particular, the length of $C$ is divisible by $6$.

There are two drawings of the three-path $ab_+c_-a$, both irreducible, as shown in Figure~\ref{figure:ab+c-a1}.
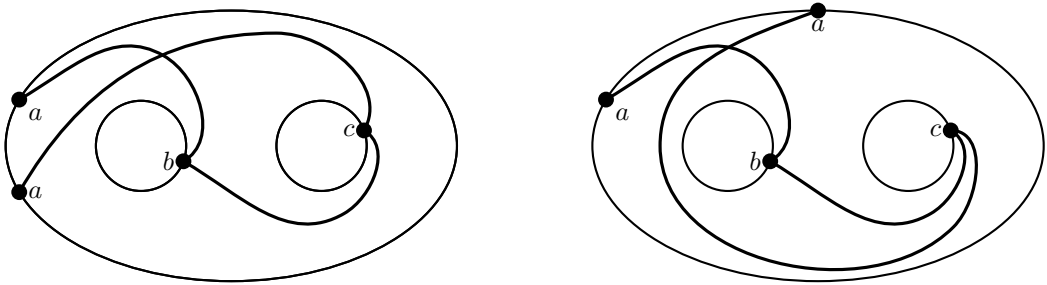
\begin{figure}[h]
\centering
\begin{tikzpicture}[scale=0.6,>=triangle 45]
\draw[thick]  (0,0) ellipse (5 and 3); \draw[thick]  (-2,0) circle (1); \draw[thick]  (2,0) circle (1);
\foreach \y in {0,13} {
\draw[thick]  (\y,0) ellipse (5 and 3); \draw[thick]  (\y-2,0) circle (1); \draw[thick]  (\y+2,0) circle (1);
\node[coordinate] (A1) at ({\y + 5*cos((8*pi/9) r)},{3*sin((8*pi/9) r)}) [label=-30:$a$] {};
\node[coordinate] (B1) at ({\y -2+cos((-pi/9) r)},{sin((-pi/9) r)}) [label=180:$b$] {};
\node[coordinate] (C1) at ({\y + 2+cos((pi/9) r)},{sin((pi/9) r)}) [label=180:$c$] {};
\draw [very thick] (A1) to[out=30,in=150] (\y -1.5,2) to[out=-30,in=30] (B1);
\draw [very thick] (B1) to[out=-30,in=-150] (\y + 2.5,-1.5) to[out=30,in=-30] (C1);
\ifthenelse{\y = 0}
    {\node[coordinate] (A2) at ({\y + 5*cos((10*pi/9) r)},{3*sin((10*pi/9) r)}) [label=0:$a$] {};
    \draw [very thick] (A2) to[out=60,in=180] (\y + 1,2.5) to[out=0,in=60] (C1)}
    {\node[coordinate] (A2) at (\y,3) [label=-90:$a$] {};
    \draw [very thick] (C1) to [out=0,in=45] (\y+3,-1.8) to [out=-135,in=-90] (\y-3.5,0) to [out=90,in=-160] (A2)}
  ;
\foreach \x in {A1,A2,B1,C1} {\fill (\x) circle (5.0pt);}
}
\end{tikzpicture}
\caption{Two paths $ab_+c_-a$.}
\label{figure:ab+c-a1}
\end{figure}
They differ by the orientation of the crossing of the edges $ab$ and $ca$. If we change the direction on $C$ and swap the letters $b$ and $c$, the subword $ab_+c_-a$ does not change. By reflection, a similar comment applies to subwords of the form $ab_-c_+a$. Hence the whole word $W$ is unchanged, with all the subscripts, but the orientation of the crossings of the edges $ab$ and $ca$ are reversed. We therefore lose no generality by assuming that the subword $ab_+c_-a$ is represented by the three-path on the left in Figure~\ref{figure:ab+c-a1}. We can then uniquely, up to isotopy, add an edge $ca$ to the starting vertex $a$, as on the left in Figure~\ref{figure:ca-b+c-a}, which produces the four-paths corresponding to the subword $ca_-b_+c_-a$. Furthermore, up to isotopy and a Reidemeister move, we can uniquely add an edge $bc$ to the starting vertex $c$, as on the right in Figure~\ref{figure:ca-b+c-a}. We get the five-paths corresponding to the subword $bc_+a_-b_+c_-a$.
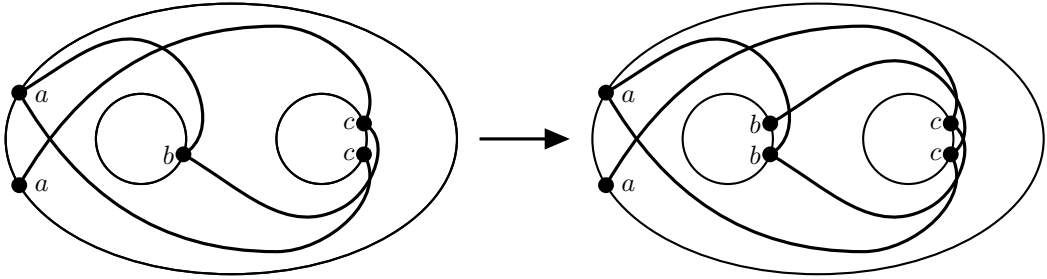
\begin{figure}[h]
\centering
\begin{tikzpicture}[scale=0.6,>=triangle 45]
\draw[thick]  (0,0) ellipse (5 and 3); \draw[thick]  (-2,0) circle (1); \draw[thick]  (2,0) circle (1);
\foreach \y in {0,13} {
\draw[thick]  (\y,0) ellipse (5 and 3); \draw[thick]  (\y-2,0) circle (1); \draw[thick]  (\y+2,0) circle (1);
\node[coordinate] (A1) at ({\y + 5*cos((8*pi/9) r)},{3*sin((8*pi/9) r)}) [label={[xshift=0.3cm, yshift=-0.25cm]:$a$}] {};
\node[coordinate] (B1) at ({\y -2+cos((-pi/9) r)},{sin((-pi/9) r)}) [label=180:$b$] {};
\node[coordinate] (C1) at ({\y + 2+cos((pi/9) r)},{sin((pi/9) r)}) [label=180:$c$] {};
\node[coordinate] (A2) at ({\y + 5*cos((10*pi/9) r)},{3*sin((10*pi/9) r)}) [label={[xshift=0.3cm, yshift=-0.2cm]:$a$}] {};
\node[coordinate] (C2) at ({\y+2+cos((-pi/9) r)},{sin((-pi/9) r)}) [label=180:$c$] {};
\draw [very thick] (A1) to[out=30,in=150] (\y -1.5,2) to[out=-30,in=30] (B1);
\draw [very thick] (B1) to[out=-30,in=-150] (\y + 2.5,-1.5) to[out=30,in=-30] (C1);
\draw [very thick] (A2) to[out=60,in=180] (\y + 1,2.5) to[out=0,in=60] (C1);
\draw [very thick] (A1) to[out=-60,in=180] (\y+1,-2.5) to[out=0,in=-60] (C2);
\foreach \x in {A1,A2,B1,C1,C2} {\fill (\x) circle (5.0pt);}
}
\node[coordinate] (B2) at ({13 -2+cos((pi/9) r)},{sin((pi/9) r)}) [label=180:$b$] {};
\draw [very thick] (B2) to[out=30,in=150] (13+2.5,1.5) to[out=-30,in=30] (C2); \fill (B2) circle (5.0pt);
\draw[->, very thick] (5.5,0) -- (7.5,0);
\end{tikzpicture}
\caption{The four-path paths $ca_-b_+c_-a$ and the five-path $bc_+a_-b_+c_-a$.}
\label{figure:ca-b+c-a}
\end{figure}
One possibility for completing the cycle would be to now join the degree one vertices $a$ and $b$ of the five-path by an edge. This can be done uniquely up to isotopy and produces an irreducible pants thrackle drawing of a six-cycle corresponding to the word $W=b_-c_+a_-b_+c_-a_+$, as in Figure~\ref{figure:sixcycle}. Any other such drawing is equivalent to that up to isotopy and Reidemeister moves (which were possible at the intermediate steps of our construction).

Otherwise, we can extend the five-path to a six-path corresponding to the subword $ab_-c_+a_-b_+c_-a$ by adding an edge $ab$ at the start. The resulting six-path is equivalent, up to isotopy and Reidemeister moves, to the one on the left in Figure~\ref{figure:abcabca}.
\begin{figure}[h]
\centering
\begin{tikzpicture}[scale=0.6,>=triangle 45]
\foreach \y in {0,13} {
\draw[thick]  (\y,0) ellipse (5 and 3); \draw[thick]  (\y-2,0) circle (1); \draw[thick]  (\y+2,0) circle (1);
\node[coordinate] (A1) at ({\y + 5*cos((8*pi/9) r)},{3*sin((8*pi/9) r)}) [label={[xshift=0.3cm, yshift=-0.25cm]:$a$}] {};
\node[coordinate] (B1) at ({\y -2+cos((-pi/9) r)},{sin((-pi/9) r)}) [label=180:$b$] {};
\node[coordinate] (C1) at ({\y + 2+cos((pi/9) r)},{sin((pi/9) r)}) [label=180:$c$] {};
\node[coordinate] (A2) at ({\y + 5*cos((10*pi/9) r)},{3*sin((10*pi/9) r)}) [label={[xshift=0.3cm, yshift=-0.2cm]:$a$}] {};
\node[coordinate] (C2) at ({\y+2+cos((-pi/9) r)},{sin((-pi/9) r)}) [label=180:$c$] {};
\node[coordinate] (B2) at ({\y -2+cos((pi/9) r)},{sin((pi/9) r)}) [label=180:$b$] {};
\node[coordinate] (A3) at ({\y - 5},0) [label={[xshift=0.3cm, yshift=-0.25cm]:$a$}] {};
\draw [very thick] (B1) to[out=-30,in=-150] (\y + 2.5,-1.5) to[out=30,in=-30] (C1);
\draw [very thick] (A2) to[out=60,in=180] (\y + 1,2.5) to[out=0,in=60] (C1);
\draw [very thick] (B2) to[out=30,in=150] (\y+2.5,1.5) to[out=-30,in=30] (C2);
\draw [very thick] (A1) to[out=30,in=150] (\y -1.5,2) to[out=-30,in=30] (B1);
    \draw [very thick] (A1) to[out=-60,in=180] (\y+1,-2.5) to[out=0,in=-60] (C2);
    \draw [very thick] (A3) to[out=-30,in=-150] (\y-1.5,-2) to[out=30,in=-30] (B2);
\foreach \x in {A1,A2,A3,B1,B2,C1,C2} {\fill (\x) circle (5.0pt);}
}
\node[coordinate] (X) at (13,-1.5) {}; \draw [thick] ($ (X) + (-0.02,-0.15) $) circle (4.0pt);
\draw [very thick] (A3) to[out=45,in=180] (11,1.5) to[out=0,in=75] (X);
\draw[->, very thick] (5.5,0) -- (7.5,0);
\end{tikzpicture}
\caption{The six-path $ab_-c_+a_-b_+c_-a$ cannot be extended to a seven-path $ca_+b_-c_+a_-b_+c_-a$.}
\label{figure:abcabca}
\end{figure}
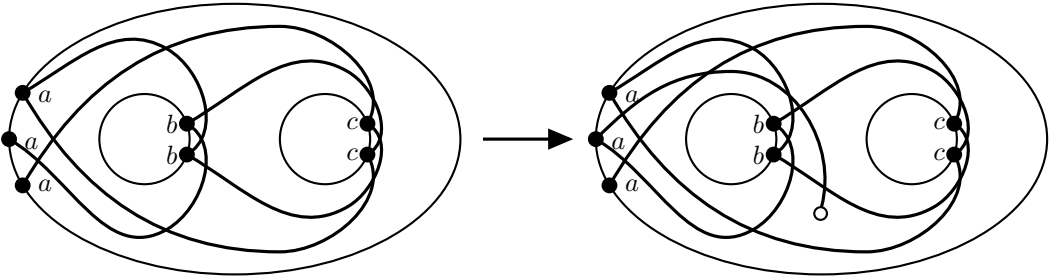
But then no edge $ca$ (with the correct orientation at $a$) can be added at the start of the six-path: up to isotopy and Reidemeister moves, the only edge we can add does not start at $c$, as on the right in Figure~\ref{figure:abcabca}. This completes the proof of the Proposition. 
\end{proof}
}

By the Proposition, if an even cycle of length greater than $6$ has a pants thrackle drawing, then it must be reducible. Hence, to prove assertion~\eqref{it:pantseven} of Theorem~\ref{t:pants}, it suffices to show that the pants thrackle drawing of the six-cycle in Figure~\ref{figure:sixcycle} (or Reidemeister equivalent to it) admits no edge insertion such that the resulting thrackle drawing of the eight-cycle is again a pants thrackle drawing. One possible way is to consider all edge insertions following the procedure in Section~\ref{ss:ir}. But as the resulting thrackles are sufficiently small, all these cases can be treated by computer. Using the algorithm given in the end of Section~3 of \cite{MNajc} we found that up to isotopy and Reidemeister moves, there exist exactly three thrackled eight-cycles; they are shown in Figure~\ref{figure:alleight}. Each of them is obtained by edge insertion in a thrackled six-cycle and belongs to class $T_4$, but none of them is a pants thrackle.
\begin{figure}[h]
\centering
\begin{tikzpicture}[scale=0.6,>=triangle 45]
\foreach \y in {0,8,16} {
\begin{scope}[rotate=90]
\coordinate (A1) at ({5*cos((10*pi/9) r)},{\y+3*sin((10*pi/9) r)});
\coordinate (A2) at ({5*cos((8*pi/9) r)},{\y + 3*sin((8*pi/9) r)});
\coordinate (B1) at ({-2+cos((-pi/5) r)},{\y + sin((-pi/5) r)});
\coordinate (B2) at ({-2+cos((pi/5) r)},{\y + sin((pi/5) r)});
\coordinate (C2) at ({1.3+cos((pi/5) r)},{\y+sin((pi/5) r)});
\coordinate (C1) at ({1.3+cos((-pi/5) r)},{\y+sin((-pi/5) r)});
\foreach \x in {A1,A2,B1,B2,C1,C2} {\fill (\x) circle (5.0pt);}
\draw [very thick] (A1) to[out=-30,in=-150] ({-1.5},\y-2) to[out=30,in=-30] (B2);
\draw [very thick] (B1) to[out=-30,in=-150] ({2.5},\y-1.5) to[out=30,in=-30] (C2);
\draw [very thick] (B2) to[out=30,in=150] ({2.5},\y+1.5) to[out=-30,in=30] (C1);
\draw [very thick] (A1) to[out=60,in=180] ({1},\y+2.5) to[out=0,in=60] (C2);
\draw [very thick] (A2) to[out=30,in=150] ({-1.5},\y+2) to[out=-30,in=30] (B1);
\ifthenelse{\y = 16}
    {\node[coordinate] (A21) at ($(A2)+(0.7,0.1)$) {}; \fill (A21) circle (5.0pt);
    \coordinate (C12) at ({1.3+cos((-pi/5) r)},\y); \fill (C12) circle (5.0pt);
    \draw [very thick] (A2) to [out=-60,in=210] (2,\y-2.2) to [out=25,in=-60] (C12);
    \draw [very thick] (A21) to [out=105,in=120] ($(A2)+(-0.4,-0.2)$) to [out=-60,in=180] (0.8,\y-3) to [out=0,in=-90] (C1);
    \draw [very thick] (A21) to [out=105,in=120] ($(A2)+(-0.6,-0.2)$) to [out=-60,in=180] (0.8,\y-3.2) to [out=0,in=-90] ($(C1)+(0.6,-0.4)$) to [out=90,in=30] (C12);
    }
    {
    \ifthenelse{\y = 8}
    {\node[coordinate] (A21) at ($(A2)+(0.2,0.5)$) {}; \fill (A21) circle (5.0pt);
    \coordinate (C12) at ($(C1)+ (0.5,0)$); \fill (C12) circle (5.0pt);
    \draw [very thick] (A21) to [out=-60,in=210] (2.2,\y-2.2) to [out=25,in=0] (C1);
    \draw [very thick] (A2) to [out=-60,in=200] (0.6,\y-1.5) to [out=20,in=180] ($0.5*(C1)+0.5*(C2)$) to [out=0,in=-90] (C12);
    \draw [very thick] (A21) to [out=-40,in=200] (0.6,\y-1.1) to [out=20,in=180] ($0.5*(C1)+0.5*(C2)+(0,0.3)$) to [out=0,in=-100] (C12);
    }
    {\node[coordinate] (A21) at ($(A2)+(0.2,0.5)$) {}; \fill (A21) circle (5.0pt);
    \coordinate (C12) at (0.5,2); \fill (C12) circle (5.0pt);
    \draw [very thick] (A21) to [out=-60,in=210] (2.2,\y-2.2) to [out=25,in=0] (C1);
    \draw [very thick] (A2) to [out=-60,in=200] (0.6,\y-1.5) to [out=20,in=-40] (C12);
    \draw [very thick] (A21) to [out=-40,in=200] (0.6,\y-1.1) to [out=20,in=-60] (C12);
    }
    }
  ;
\end{scope}
}
\end{tikzpicture}
\caption{All thrackled eight-cycles up to Reidemeister equivalency.}
\label{figure:alleight}
\end{figure}
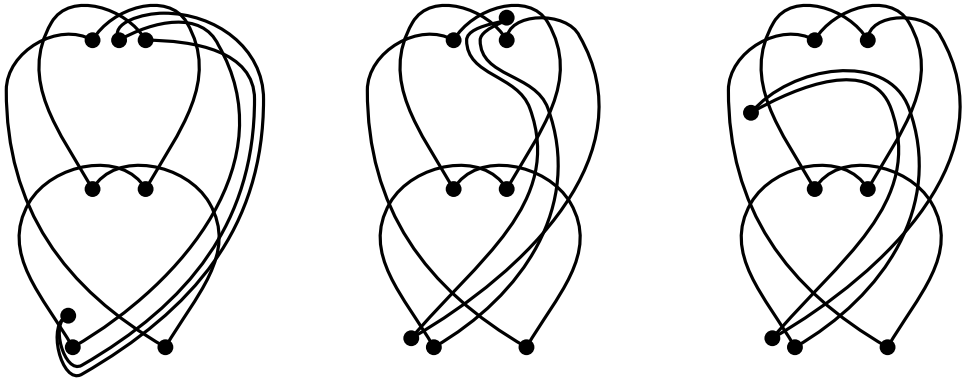
This proves assertion~\eqref{it:pantseven} of Theorem~\ref{t:pants}.

It remains to prove assertion~\eqref{it:pantsC}. By Lemma~\ref{l:TCTd}, it suffices to show that if $G$ is either a theta-graph or a dumbbell, then it admits no pants thrackle drawing. We also know from Lemma~\ref{l:TCTd} that in both cases, $G$ contains an even cycle which by assertion~\eqref{it:pantseven} must be a six-cycle whose thrackle drawing is Reidemeister equivalent to the one in Figure~\ref{figure:sixcycle}. 

The proof goes as follows: we explicitly construct pants thrackle drawings of a six-cycle with certain small trees attached to one of its vertices and first show that in a pants thrackle drawing of a three-path attached to a six-cycle, the drawing of the three-path is reducible. Repeatedly performing edge removals we get a pants thrackle drawing either of a theta-graph obtained from a six-cycle by joining two of its vertices by a path of length at most $2$, or of a dumbbell consisting of a six-cycle and some other cycle joined by a path of length at most $2$. The resulting theta-graphs are very small, and from \cite{FP2011, MNajc} we know that they admit no thrackle drawing at all, and in particular, no pants thrackle drawing (the latter fact will also be confirmed in the course of the proof). Every resulting dumbbell contains one of two subgraphs obtained from the six-cycle by attaching a small tree, as in Figure~\ref{figure:sixtree}.
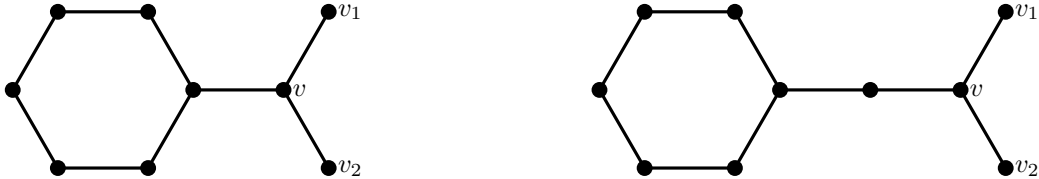
\begin{figure}[h]
\centering
\begin{tikzpicture}[scale=0.6,>=triangle 45]
\foreach \y in {0,13} {
\node[coordinate] (A1) at (\y + 2,0) {};
\node[coordinate] (A2) at ({\y + 2*cos((pi/3) r)},{2*sin((pi/3) r)}) {};
\node[coordinate] (A3) at ({\y + 2*cos((2*pi/3) r)},{2*sin((2*pi/3) r)}) {};
\node[coordinate] (A4) at ({\y - 2},0) {};
\node[coordinate] (A5) at ({\y + 2*cos((-2*pi/3) r)},{2*sin((-2*pi/3) r)}) {};
\node[coordinate] (A6) at ({\y + 2*cos((-pi/3) r)},{2*sin((-pi/3) r)}) {};
\foreach \x in {A1,A2,A3,A4,A5,A6} {\fill (\x) circle (5.0pt);}
\draw [very thick] (A1) -- (A2) -- (A3) -- (A4) -- (A5) -- (A6) -- (A1);
    \ifthenelse{\y = 0}
    {\node[coordinate] (B1) at ({\y + 4},0) [label=0:$v$] {}; \fill (B1) circle (5.0pt);
    \node[coordinate] (C1) at ({\y + 4 + 2*cos((pi/3) r)},{2*sin((pi/3) r)}) [label=0:$v_1$] {}; \fill (C1) circle (5.0pt);
    \node[coordinate] (C2) at ({\y + 4 + 2*cos((-pi/3) r)},{2*sin((-pi/3) r)})  [label=0:$v_2$] {}; \fill (C2) circle (5.0pt);
    \draw [very thick] (C1) -- (B1) -- (C2); \draw [very thick] (A1) -- (B1);
    }
    {\node[coordinate] (B1) at ({\y + 4},0) {}; \fill (B1) circle (5.0pt);
    \node[coordinate] (B2) at ({\y + 6},0) [label=0:$v$] {}; \fill (B2) circle (5.0pt);
    \node[coordinate] (C1) at ({\y + 6 + 2*cos((pi/3) r)},{2*sin((pi/3) r)}) [label=0:$v_1$] {}; \fill (C1) circle (5.0pt);
    \node[coordinate] (C2) at ({\y + 6 + 2*cos((-pi/3) r)},{2*sin((-pi/3) r)})  [label=0:$v_2$] {}; \fill (C2) circle (5.0pt);
    \draw [very thick] (C1) -- (B2) -- (C2); \draw [very thick] (A1) -- (B1) -- (B2);
    }
}
\end{tikzpicture}
\caption{The six-path cycle with a tree attached.}
\label{figure:sixtree}
\end{figure}
We show that for a pants thrackle drawing of each of these two subgraphs, to at least one of the two vertices $v_1, v_2$, it is not possible to attach another edge so that the resulting drawing is a pants thrackle drawing.

We start with the pants thrackle drawing of the six-cycle and attach a path to one of its vertices. By cyclic symmetry, we can choose any vertex to attach a path. Moreover, from the arguments in Section~\ref{ss:R} it follows that Reidemeister moves on the original six-cycle and on the intermediate steps of adding edges will result in a Reidemeister equivalent drawing in the end. So we can attach a path edge-by-edge choosing one of Reidemeister equivalent drawings arbitrarily at each step.

Up to isotopy and Reidemeister moves, there are two ways to attach an edge to a vertex of the drawing of the six-cycle, as in Figure~\ref{figure:sixplusone}. Note that the second endpoint of this edge is not one of the vertices of the six-cycle (so that no theta-graph obtained by joining two vertices of a six-cycle by an edge admits a pants thrackle drawing) and that in the two cases shown in Figure~\ref{figure:sixplusone}, it lies on different boundary components of $P$. It follows that renaming $b$ and $c$ and changing the direction on the cycle and the orientation on the plane, we obtain two Reidemeister equivalent drawings. We continue with the one on the left in Figure~\ref{figure:sixplusone} and attach another edge at the vertex of degree $1$.
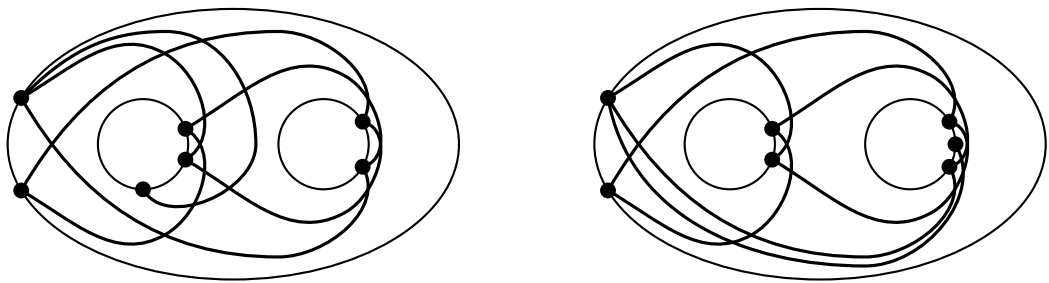
\begin{figure}[h]
\centering
\begin{tikzpicture}[scale=0.6,>=triangle 45]
\foreach \y in {0,13} {
\draw[thick]  (\y,0) ellipse (5 and 3); \draw[thick]  (\y-2,0) circle (1); \draw[thick]  (\y+2,0) circle (1);
\coordinate (A1) at ({\y+5*cos((10*pi/9) r)},{3*sin((10*pi/9) r)});
\node[coordinate] (A2) at ({\y + 5*cos((8*pi/9) r)},{3*sin((8*pi/9) r)}) {};
\node[coordinate] (B1) at ({\y -2+cos((-pi/9) r)},{sin((-pi/9) r)}) {};
\coordinate (B2) at ({\y-2+cos((pi/9) r)},{sin((pi/9) r)});
\coordinate (C2) at ({\y+2+cos((pi/6) r)},{sin((pi/6) r)});
\coordinate (C1) at ({\y+2+cos((-pi/6) r)},{sin((-pi/6) r)});
\foreach \x in {A1,A2,B1,B2,C1,C2} {\fill (\x) circle (5.0pt);}
\draw [very thick] (A1) to[out=-30,in=-150] ({\y-1.5},-2) to[out=30,in=-30] (B2);
\draw [very thick] (B1) to[out=-30,in=-150] ({\y+2.5},-1.5) to[out=30,in=-15] (C2);
\draw [very thick] (A2) to[out=-60,in=180] ({\y+1},-2.5) to[out=0,in=-60] (C1);
\draw [very thick] (A2) to[out=30,in=150] ({\y-1.5},2) to[out=-30,in=30] (B1);
\draw [very thick] (B2) to[out=30,in=150] ({\y+2.5},1.5) to[out=-30,in=15] (C1);
\draw [very thick] (A1) to[out=60,in=180] ({\y+1},2.5) to[out=0,in=60] (C2);
\ifthenelse{\y = 0}
    {
    \coordinate (B3) at ({\y-2},-1); \fill (B3) circle (5.0pt);
    \draw [very thick] (A2) to [out=45,in=180] ({\y-1.5},2.5) to [out=0,in=90] ({\y+0.5},0) to [out=-90,in=-60] (B3);
    }
    {
    \coordinate (C3) at ({\y+3},0); \fill (C3) circle (5.0pt);
    \draw [very thick] (A2) to[out=-80,in=180] ({\y+1},-2.7) to[out=0,in=-60] (C3);
    }
  ;
}
\end{tikzpicture}
\caption{Six-cycle with an edge attached.}
\label{figure:sixplusone}
\end{figure}
This can be done uniquely up to isotopy and Reidemeister equivalence resulting in the drawing as on the left in Figure~\ref{figure:sixplustwo}. Again, the second endpoint of the attached edge cannot be one of the vertices of the six-cycle (so that no theta-graph obtained by joining two vertices of a six-cycle by a two-path admits a pants thrackle drawing). Then we can attach another edge at that vertex. This can be done uniquely up to isotopy and Reidemeister equivalence, as on the right in Figure~\ref{figure:sixplustwo}. 
\begin{figure}[h]
\centering
\begin{tikzpicture}[scale=0.6,>=triangle 45]
\foreach \y in {0,13} {
\draw[thick]  (\y,0) ellipse (5 and 3); \draw[thick]  (\y-2,0) circle (1); \draw[thick]  (\y+2,0) circle (1);
\coordinate (A1) at ({\y+5*cos((10*pi/9) r)},{3*sin((10*pi/9) r)});
\node[coordinate] (A2) at ({\y + 5*cos((8*pi/9) r)},{3*sin((8*pi/9) r)}) {};
\node[coordinate] (B1) at ({\y -2+cos((-pi/9) r)},{sin((-pi/9) r)}) {};
\coordinate (B2) at ({\y-2+cos((pi/9) r)},{sin((pi/9) r)});
\coordinate (C2) at ({\y+2+cos((pi/6) r)},{sin((pi/6) r)});
\coordinate (C1) at ({\y+2+cos((-pi/6) r)},{sin((-pi/6) r)});
\foreach \x in {A1,A2,B1,B2,C1,C2} {\fill (\x) circle (5.0pt);}
\draw [very thick] (A1) to[out=-30,in=-150] ({\y-1.5},-2) to[out=30,in=-30] (B2);
\draw [very thick] (B1) to[out=-30,in=-150] ({\y+2.5},-1.5) to[out=30,in=-15] (C2);
\draw [very thick] (A2) to[out=-60,in=180] ({\y+1},-2.5) to[out=0,in=-60] (C1);
\draw [very thick] (A2) to[out=30,in=150] ({\y-1.5},2) to[out=-30,in=30] (B1);
\draw [very thick] (B2) to[out=30,in=150] ({\y+2.5},1.5) to[out=-30,in=15] (C1);
\draw [very thick] (A1) to[out=60,in=180] ({\y+1},2.5) to[out=0,in=60] (C2);
\node[coordinate] (B3) at ({\y-2+cos((pi/3) r)},{sin((-pi/3) r)}) {}; \fill (B3) circle (5.0pt);
\draw [very thick] (A2) to [out=45,in=180] ({\y-1.5},2.5) to [out=0,in=90] ({\y+0.5},0) to [out=-90,in=-60] (B3);
\coordinate (A3) at ({\y+5*cos((19*pi/18) r)},{3*sin((19*pi/18) r)}); \fill (A3) circle (5.0pt);
\draw [very thick] (A3) to [out=45,in=180] ({\y-1.5},2.2) to [out=0,in=90] ({\y+0.2},0) to [out=-90,in=-30] (B3);
\ifthenelse{\y = 0}{}
    {
    \coordinate (B4) at ({\y-2+cos((-pi/2-pi/9) r)},{sin((-pi/2-pi/9) r)});\fill (B4) circle (5.0pt);
    \draw [very thick] (A3) to [out=60,in=180] ({\y-1.2},2.7) to [out=0,in=90] ({\y+0.7},0) to [out=-90,in=-45] (B4);
    }
  ;
}
\draw[->, very thick] (5.5,0) -- (7.5,0);
\end{tikzpicture}
\caption{Six-cycle with two- and three-paths attached.}
\label{figure:sixplustwo}
\end{figure}
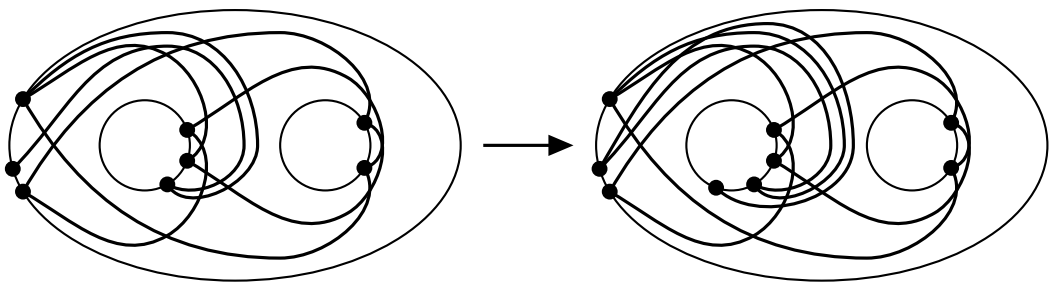
But if the two vertices other than the endpoints in the so attached three-path have degree $2$ in $G$, then the three-path is reducible by Lemma~\ref{lemma:triangle}. Now if $G$ is a theta-graph, then by repeatedly performing edge removals we obtain a pants thrackle drawing of a theta-graph obtained by joining two vertices of a six-cycle by a path of length at most $2$, which is impossible, as we have shown above. If $G$ is a dumbbell, then by repeatedly performing edge removals we obtain a pants thrackle drawing of a dumbbell consisting of the six-cycle and a cycle $C'$, with a vertex of the six-cycle joined to the vertex $v$ of $C'$ by either an edge or a two-path. Such a dumbbell contains one of the two subgraphs given in Figure~\ref{figure:sixtree}. So it remains to deal with these two cases.

The vertex $v$ has degree $3$ in $G$, so we have to attach two edges to it. In the first case, we start with the drawing on the left in Figure~\ref{figure:sixplusone} and attach two edges to the vertex $v$. We obtain a unique drawing, up to isotopy and Reidemeister moves, as on the right in Figure~\ref{figure:sixplusonev}. But then no edge can be attached to the vertex $a_2$ in such a way that the resulting drawing is a pants thrackle drawing.
\begin{figure}[h]
\centering
\begin{tikzpicture}[scale=0.6,>=triangle 45]
\foreach \y in {0,14} {
\draw[thick]  (\y,0) ellipse (5 and 3); \draw[thick]  (\y-2,0) circle (1); \draw[thick]  (\y+2,0) circle (1);
\coordinate (A1) at ({\y+5*cos((10*pi/9) r)},{3*sin((10*pi/9) r)});
\node[coordinate] (A2) at ({\y + 5*cos((8*pi/9) r)},{3*sin((8*pi/9) r)}) {};
\node[coordinate] (B1) at ({\y -2+cos((-pi/9) r)},{sin((-pi/9) r)}) {};
\coordinate (B2) at ({\y-2+cos((pi/9) r)},{sin((pi/9) r)});
\coordinate (C2) at ({\y+2+cos((pi/6) r)},{sin((pi/6) r)});
\coordinate (C1) at ({\y+2+cos((-pi/6) r)},{sin((-pi/6) r)});
\foreach \x in {A1,A2,B1,B2,C1,C2} {\fill (\x) circle (5.0pt);}
\draw [very thick] (A1) to[out=-30,in=-150] ({\y-1.5},-2) to[out=30,in=-30] (B2);
\draw [very thick] (B1) to[out=-30,in=-150] ({\y+2.5},-1.5) to[out=30,in=-15] (C2);
\draw [very thick] (A2) to[out=-60,in=180] ({\y+1},-2.5) to[out=0,in=-60] (C1);
\draw [very thick] (A2) to[out=30,in=150] ({\y-1.5},2) to[out=-30,in=30] (B1);
\draw [very thick] (B2) to[out=30,in=150] ({\y+2.5},1.5) to[out=-30,in=15] (C1);
\draw [very thick] (A1) to[out=60,in=180] ({\y+1},2.5) to[out=0,in=60] (C2);
\node[coordinate] (B3) at ({\y-2},-1) [label=90:$v$] {}; \fill (B3) circle (5.0pt);
\draw [very thick] (A2) to [out=45,in=180] ({\y-1.5},2.5) to [out=0,in=90] ({\y+0.5},0) to [out=-90,in=-60] (B3);
\ifthenelse{\y = 0}
    {}
    {\node[coordinate] (A3) at ({\y+5*cos((19*pi/18) r)},{3*sin((19*pi/18) r)}) [label=180:$a_1$] {}; \fill (A3) circle (5.0pt);
    \draw [very thick] (A3) to [out=45,in=180] ({\y-1.5},1.8) to [out=0,in=90] ({\y-0.3},0) to [out=-90,in=0] (B3);
    \node[coordinate] (A4) at ({\y+5*cos((35*pi/36) r)},{3*sin((35*pi/36) r)}) [label=180:$a_2$] {}; \fill (A4) circle (5.0pt);
    \draw [very thick] (A4) to [out=45,in=180] ({\y-1.5},2.1) to [out=0,in=90] ({\y},0) to [out=-90,in=-40] (B3);
    }
  ;
}
\draw[->, very thick] (5.75,0) -- (7.75,0);
\end{tikzpicture}
\caption{Pants drawing of the graph on the left in Figure~\ref{figure:sixtree}.} 
\label{figure:sixplusonev}
\end{figure}
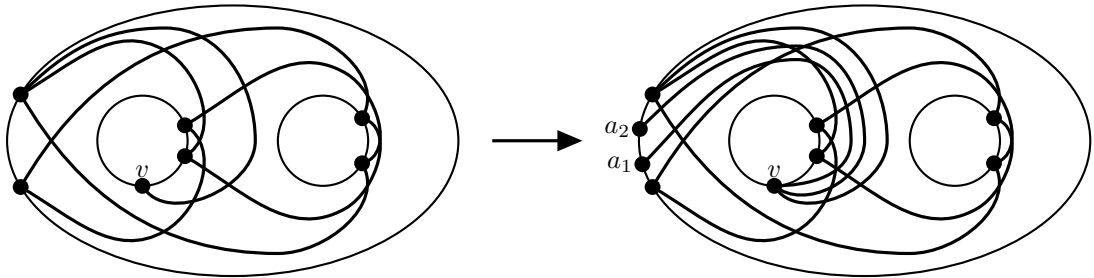

Similarly, in the second case, we start with the drawing on the left in Figure~\ref{figure:sixplustwo} and attach two edges to the vertex $v$. We obtain a unique drawing, up to isotopy and Reidemeister moves, as on the right in Figure~\ref{figure:sixplustwov}. But then no edge can be attached to the vertex $b_1$ in such a way that the resulting drawing is a pants thrackle drawing.

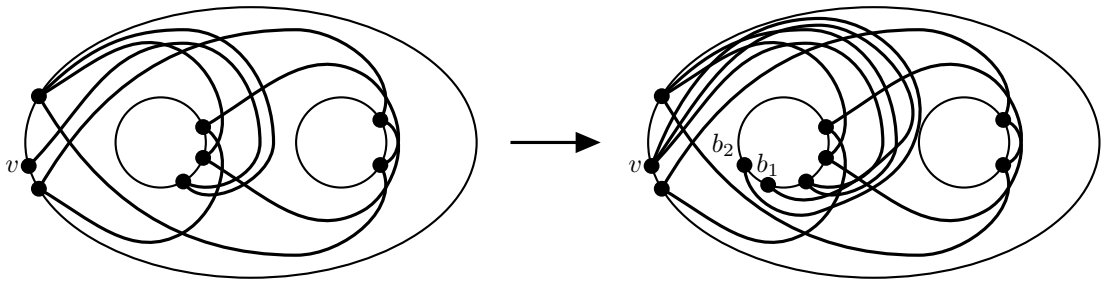
\begin{figure}[h]
\centering
\begin{tikzpicture}[scale=0.6,>=triangle 45]
\foreach \y in {0,13.8} {
\pgfmathparse{abs(\y) < 0.001 ? int(1) : int(0)}
\draw[thick]  (\y,0) ellipse (5 and 3); \draw[thick]  (\y-2,0) circle (1); \draw[thick]  (\y+2,0) circle (1);
\coordinate (A1) at ({\y+5*cos((10*pi/9) r)},{3*sin((10*pi/9) r)});
\node[coordinate] (A2) at ({\y + 5*cos((8*pi/9) r)},{3*sin((8*pi/9) r)}) {};
\node[coordinate] (B1) at ({\y -2+cos((-pi/9) r)},{sin((-pi/9) r)}) {};
\coordinate (B2) at ({\y-2+cos((pi/9) r)},{sin((pi/9) r)});
\coordinate (C2) at ({\y+2+cos((pi/6) r)},{sin((pi/6) r)});
\coordinate (C1) at ({\y+2+cos((-pi/6) r)},{sin((-pi/6) r)});
\foreach \x in {A1,A2,B1,B2,C1,C2} {\fill (\x) circle (5.0pt);}
\draw [very thick] (A1) to[out=-30,in=-150] ({\y-1.5},-2) to[out=30,in=-30] (B2);
\draw [very thick] (B1) to[out=-30,in=-150] ({\y+2.5},-1.5) to[out=30,in=-15] (C2);
\draw [very thick] (A2) to[out=-60,in=180] ({\y+1},-2.5) to[out=0,in=-60] (C1);
\draw [very thick] (A2) to[out=30,in=150] ({\y-1.5},2) to[out=-30,in=30] (B1);
\draw [very thick] (B2) to[out=30,in=150] ({\y+2.5},1.5) to[out=-30,in=15] (C1);
\draw [very thick] (A1) to[out=60,in=180] ({\y+1},2.5) to[out=0,in=60] (C2);
\node[coordinate] (B3) at ({\y-2+cos((pi/3) r)},{sin((-pi/3) r)}) {}; \fill (B3) circle (5.0pt);
\draw [very thick] (A2) to [out=45,in=180] ({\y-1.5},2.5) to [out=0,in=90] ({\y+0.5},0) to [out=-90,in=-60] (B3);
\node[coordinate] (A3) at ({\y+5*cos((19*pi/18) r)},{3*sin((19*pi/18) r)}) [label=180:$v$] {}; \fill (A3) circle (5.0pt);
\draw [very thick] (A3) to [out=45,in=180] ({\y-1.5},2.2) to [out=0,in=90] ({\y+0.2},0) to [out=-90,in=-30] (B3);
\ifnum\pgfmathresult=0
    {
    \node[coordinate] (B4) at ({\y-2+cos((-pi/2-pi/9) r)},{sin((-pi/2-pi/9) r)}) [label=90:$b_1$] {};\fill (B4) circle (5.0pt);
    \draw [very thick] (A3) to [out=55,in=180] ({\y-1.2},2.6) to [out=0,in=45] ({\y+0.6},0) to [out=-135,in=-45] (B4);
    \node[coordinate] (B5) at ({\y-2+cos((-pi/2-3*pi/9) r)},{sin((-pi/2-3*pi/9) r)}) [label=135:$b_2$] {};\fill (B5) circle (5.0pt);
    \draw [very thick] (A3) to [out=70,in=180] ({\y-1.2},2.75) to [out=0,in=80] ({\y+1},0.4) to [out=-100,in=20] ({\y-1.25},-1.55) to [out=-160,in=-90] (B5);
    }
  \fi
}
\draw[->, very thick] (5.75,0) -- (7.75,0);
\end{tikzpicture}
\caption{Pants drawing of the graph on the right in Figure~\ref{figure:sixtree}.}
\label{figure:sixplustwov}
\end{figure}


This completes the proof of Theorem~\ref{t:pants}.

\acknowledgements
\label{sec:ack}
We express our deep gratitude to Grant Cairns for his generous contribution to this paper, at all the stages, from mathematics to presentation.

We are thankful to the reviewer for their kind permission to include a brief description of the ideas underlining the proof borrowed from their report.

\nocite{*}
\bibliographystyle{alpha}
\bibliography{cylnew}
\label{sec:biblio}

\end{document}